\documentclass{conm-p-l}

\usepackage{amssymb}

\usepackage[pdftex]{graphicx}

\usepackage{latexsym}
\usepackage{mathabx}
\usepackage{amscd}
\usepackage{enumerate}
\usepackage{tikz}
\usepackage{comment}
\usepackage{soul}
\usepackage{listings}
\usepackage{courier}
\usepackage{tikz-cd}
\usepackage{xcolor}
\usepackage{MnSymbol}
\usepackage{hyperref}
\usepackage{colonequals}

\newcommand{\pp}{\mathbb{P}}

\newcommand{\qq}{\mathbb{Q}}
\newcommand{\zz}{\mathbb{Z}}
\newcommand{\ff}{\mathbb{F}}

\newcommand{\C}{{\mathcal C}} 
\newcommand{\FF}{\mathbb{F}} 

\newcommand{\Res}{\operatorname{Res}}

\newcommand{\GL}{\textup{GL}}
\newcommand{\SL}{\text{SL}}

\newcommand{\GSp}{\operatorname{GSp}}
\newcommand{\PGSp}{\operatorname{PGSp}}
\newcommand{\PSp}{\operatorname{PSp}}

\newcommand{\Sp}{\text{Sp}}

\newcommand{\GU}{\textup{GU}}
\newcommand{\Sym}{\operatorname{Sym}}

\newcommand{\Jac}{\operatorname{Jac}}

\newcommand{\sq}{\operatorname{sq}}

\newcommand{\tr}{\operatorname{tr}}
\newcommand{\simi}{\operatorname{mult}}
\newcommand{\midd}{\operatorname{mid}}

\newcommand{\Gal}{\operatorname{Gal}}
\newcommand{\Frob}{\operatorname{Frob}}

\newcommand{\Aut}{\operatorname{Aut}}

\newcommand{\rad}{\operatorname{rad}}

\newcommand{\End}{\operatorname{End}}
\newcommand{\oAl}{\overline{A[\ell]}}

\newcommand{\cyc}{\operatorname{cyc}}

\newcommand{\cond}{\operatorname{cond}}
\newcommand{\mult}{\operatorname{mult}}

\DeclareMathOperator{\PossiblyNonsurjectivePrimes}{\textup{\textsf{PossiblyNonsurjectivePrimes}}}
\DeclareMathOperator{\LikelyNonsurjectivePrimes}{\textup{\textsf{LikelyNonsurjectivePrimes}}}

\renewcommand{\bar}{\overline}
\newcommand{\GSpFL}{{\mathrm{GSp}_4({\mathbb{F}_\ell})}} 

\definecolor{bsb_colour}{RGB}{0, 0, 204}

 \newcommand{\defi}[1]{\textsf{#1}}

\newtheorem{thm}{Theorem}[section]

\newtheorem{lem}[thm]{Lemma}

\newtheorem{prop}[thm]{Proposition}
\newtheorem{cor}[thm]{Corollary}

\newtheorem{algo}[thm]{Algorithm}

\theoremstyle{definition}

\newtheorem{example}[thm]{Example}

\newtheorem{test}[thm]{Test}
\newtheorem*{ack}{Acknowledgements}

\theoremstyle{remark}
\newtheorem{rem}{Remark}

\begin{document}

\title[Computing nonsurjective primes in genus $2$]{\boldmath Computing nonsurjective primes associated to Galois representations of genus $2$ curves}


\author[Banwait]{Barinder S. Banwait}
\address{Barinder S. Banwait \\
Department of Mathematics \& Statistics, Boston University, Boston, MA}
\email{barinder@bu.edu}
\urladdr{https://barinderbanwait.github.io/}

\author[Brumer]{Armand Brumer}
\address{Armand Brumer, Department of Mathematics, Fordham University, New York, NY}
\email{brumer@fordham.edu}

\author[Kim]{Hyun Jong Kim}
\address{Hyun Jong Kim, Department of Mathematics, University of Wisconsin-Madison, Madison, WI}
\email{hyunjong.kim@math.wisc.edu}
\urladdr{https://sites.google.com/wisc.edu/hyunjongkim}

\author[Klagsbrun]{Zev Klagsbrun}
\address{Zev Klagsbrun, Center for Communications Research, San Diego, CA}
\email{\href{mailto:zdklags@ccr-lajolla.org}{zdklags@ccr-lajolla.org}}

\author[Mayle]{Jacob Mayle}
\address{Jacob Mayle, Department of Mathematics, Wake Forest University, Winston-Salem, NC}
\email{maylej@wfu.edu}

\author[Srinivasan]{Padmavathi Srinivasan}
\address{Padmavathi Srinivasan, ICERM,
Providence, RI}
\email{padmavathi\_srinivasan@brown.edu}
\urladdr{https://padmask.github.io/}

\author[Vogt]{Isabel Vogt}
\address{Isabel Vogt, Department of Mathematics, Brown University, Providence, RI}
\email{ivogt.math@gmail.com}
\urladdr{https://www.math.brown.edu/ivogt/}

\date{\today}

\begin{abstract} For a genus $2$ curve $C$ over $\qq$ whose Jacobian $A$ admits only trivial geometric endomorphisms, Serre's open image theorem for abelian surfaces asserts that there are only finitely many primes $\ell$ for which the Galois action on $\ell$-torsion points of $A$ is not maximal. Building on work of Dieulefait, we give a practical algorithm to compute this finite set.  The key inputs are Mitchell's classification of maximal subgroups of $\PSp_4(\ff_\ell)$, sampling of the characteristic polynomials of Frobenius, and the Khare--Wintenberger modularity theorem. The algorithm has been submitted for integration into Sage, executed on all of the genus~$2$ curves with trivial endomorphism ring in the LMFDB, and the results incorporated into the homepage of each such curve.
\end{abstract}

\maketitle


\section{Introduction}

Let $C/\qq$ be a smooth, projective, geometrically integral curve (referred to hereafter as a \defi{nice} curve) of genus $2$, and let $A$ be its Jacobian. We assume throughout that $A$ admits no nontrivial geometric endomorphisms; that is, we assume that $\End(A_{\overline{\qq}}) = \zz$, and we refer to any such abelian variety as \defi{typical}\footnote{Abelian varieties with extra endomorphisms define a thin set (in the sense of Serre) in the moduli space $\mathcal{A}_g$ of principally polarized abelian varieties of dimension $g$ and as such are not the typically arising case.}. We also say that a nice curve is \defi{typical} if its Jacobian is typical. Let $G_{\qq} \colonequals \Gal\left(\bar{\qq}/\qq\right)$, let $\ell$ be a prime, and let $A[\ell] \colonequals A(\overline{\qq})[\ell]$ denote the $\ell$-torsion points of $A(\overline{\qq})$. Let 
\[\rho_{A, \ell} \ \colon \ G_{\qq} \to \Aut\left(A[\ell]\right)\]
denote the Galois representation on $A[\ell]$. By fixing a basis for $A[\ell]$, and observing that $A[\ell]$ admits a nondegenerate Galois-equivariant alternating bilinear form, namely the Weil pairing, we may identify the codomain of $\rho_{A,\ell}$ with the general symplectic group $\GSp_4(\ff_\ell)$.

In a letter to Vignéras \cite[Corollaire au Théorème 3]{serre_lettre}, Serre proved an open image theorem for typical abelian varieties of dimensions \(2\) or \(6\), or of odd dimensions, generalizing his celebrated open image theorem for elliptic curves \cite{Serre72}.  More precisely, the set of \defi{nonsurjective primes} $\ell$, namely those for which $\rho_{A, \ell}(G_{\qq})$ is a proper subgroup of $\GSp_{4}(\ff_{\ell})$, is finite.

In the elliptic curve case, Serre subsequently provided a conditional upper bound, in terms of the conductor of $E$, on this finite set \cite[Théorème 22]{serre81}; this bound has since been made unconditional \cite{cojocaru2005surjectivity, MR1360773}.  There are also algorithms to compute the finite set of nonsurjective primes  \cite{zywina2015surjectivity}, and practical implementations in Sage \cite{GalReps}.

Serre's open image theorem for typical abelian surfaces was made explicit by Dieulefait \cite{Dieulefait} who described an algorithm that returns a finite set of primes \emph{containing} the set of nonsurjective primes. In a different direction, Lombardo \cite[Theorem 1.3]{lombardo2016explicit} provided an upper bound for the largest nonsurjective prime involving the stable Faltings height of $A$.


In this paper we develop two algorithms that together provide the exact determination of the nonsurjective primes for $C$, yielding the main result of our paper as follows.

\begin{thm}\label{T:mainthm} Let \(C/\qq\) be a typical genus \(2\) curve whose Jacobian \(A\) has conductor \(N\).
\begin{enumerate}[\upshape (1)]
    \item\label{part1} Algorithm~\ref{A:combinedPart1algo} produces a finite list  $\PossiblyNonsurjectivePrimes(C)$ that provably contains all nonsurjective primes. 

    \item\label{part2} Given $B > 0$, Algorithm~\ref{A:combinedPart2algo} produces a sublist $\LikelyNonsurjectivePrimes(C; B)$ of  $\PossiblyNonsurjectivePrimes(C)$ that contains all the nonsurjective primes. If $B$ is sufficiently large, then the elements of $\LikelyNonsurjectivePrimes(C; B)$ are precisely the nonsurjective primes of $A$.
\end{enumerate}
\end{thm}

The two common ingredients in Algorithms~\ref{A:combinedPart1algo} and \ref{A:combinedPart2algo} are Mitchell's 1914 classification of maximal subgroups of $\PSp_4(\ff_\ell)$ \cite{Mitchell} and sampling of characteristic polynomials of Frobenius elements.  Indeed, $\rho_{A,\ell}$ is nonsurjective precisely when its image is contained in one of the proper maximal subgroups of $\GSp_4(\ff_\ell)$.  The (integral) characteristic polynomial of Frobenius at a good prime $p$ is computationally accessible since it is determined by counting points on $C$ over $\ff_{p^r}$ for $r \leq 2$. The reduction of this polynomial modulo $\ell$ gives the characteristic polynomial of the action of the Frobenius element on $A[\ell]$.  By the Chebotarev density theorem, the images of the Frobenius elements for varying primes $p$ equidistribute over the conjugacy classes of $\rho_{A, \ell}(G_{\qq})$ and hence let us explore the image.

Algorithm~\ref{A:combinedPart1algo} makes use of the fact that if the image of \(\rho_{A, \ell}\) is nonsurjective, then the characteristic polynomials of Frobenius at auxiliary primes \(p\) will be constrained modulo \(\ell\).  Using this idea, Dieulefait worked out the constraints imposed by each type of maximal subgroup for $\rho_{A,\ell}(G_\qq)$ to be contained in that subgroup.  Our Algorithm~\ref{A:combinedPart1algo} 
combines Dieulefait's conditions, with some modest improvements, to produce a finite list $\PossiblyNonsurjectivePrimes(C)$.

Algorithm~\ref{A:combinedPart2algo} then weeds out the extraneous surjective primes from the list
$\PossiblyNonsurjectivePrimes(C)$.
Given \(\ell\), we need to generate enough different elements in the image to rule out containment in any proper maximal subgroup.  The key input is a purely group-theoretic condition (Proposition~\ref{P:grpthy}) that guarantees that a subgroup is all of \(\GSp_4(\ff_\ell)\) if it contains particular types of elements.
This algorithm is probabilistic and depends on the choice of a parameter $B$ which, if sufficiently large, provably establishes nonsurjectivity. 
The parameter $B$ is a cut-off for the number of Frobenius elements (referred to as \emph{Frobenius witnesses})
that we use to sample the conjugacy classes of $\rho_{A,\ell}(G_{\qq})$.

As an illustration of the interplay between theory and practice,
analyzing the ``worst case'' run time of each step in Algorithm~\ref{A:combinedPart1algo} 
yields a new \emph{theoretical} bound, conditional on the Generalized Riemann Hypothesis (GRH), on the product of all nonsurjective primes in terms of the conductor.

\begin{thm}\label{T:mainbound}
Let \(C/\qq\) be a typical genus \(2\) curve with conductor \(N\).
Assuming the Generalized Riemann Hypothesis (GRH), we have, for any $\epsilon > 0$,
\[\displaystyle\prod_{\ell \text{ nonsurjective}} \ell \ll \exp(N^{1/2 + \epsilon}),\]
where the implied constant is absolute and effectively computable.
\end{thm}

While we believe this bound to be far from asymptotically optimal, it is the first bound in the literature expressed in terms of the (effectively computable) conductor.

\medskip

Naturally one wants to find the sufficiently large value of $B$ in Theorem~\ref{T:mainthm}\eqref{part2}, which the next result gives, conditional on GRH.

\begin{thm}\label{T:Bbound}
Let $C/\qq$ be a typical genus $2$ curve, and let $q$ be the largest nonsurjective prime for $C$. Then, assuming GRH, the set $\LikelyNonsurjectivePrimes(C; B)$ is precisely the set of nonsurjective primes of $C$ provided that 
\begin{equation}\label{eq:B_bound}
B \geq \left( 4\left[ (2q^{11} - 1)\log \rad(2 q N_A) + 22 q^{11} \log(2q) \right] + 5q^{11} + 5  \right)^2.
\end{equation}
\end{thm}

The proof of Theorem \ref{T:Bbound} involves an explicit Chebotarev bound due to Bach and Sorenson \cite{MR1355006} that assumes GRH. An unconditional version of Theorem \ref{T:Bbound} can be given using an unconditional Chebotarev result (for instance \cite{MR4472459}), though the bound for $B$ will be exponential in $q$. In addition, if we assume both GRH and the Artin Holomorphy Conjecture (AHC), then a version of Theorem \ref{T:Bbound} holds with the improved asymptotic bound $B \gg  q^{11} \log^2(q N_A)$, but without an explicit constant, see Remark~\ref{R:GRHAHC}.

Unfortunately, the bound from Theorem \ref{T:Bbound} is prohibitively large to use in practice. By way of illustration, consider the typical genus $2$ curve of smallest conductor, which has a model
\[ y^2 + (x^3 + 1)y = x^2 + x,\]
and label \href{https://www.lmfdb.org/Genus2Curve/Q/249/a/249/1}{\path{249.a.249.1}} in the \emph{$L$-functions and modular forms database} (LMFDB) \cite{lmfdb}.  The output of Algorithm~\ref{A:combinedPart1algo} is the set $\{ 2,3,5,7,83\}$. Applying Algorithm~\ref{A:combinedPart2algo} with $B = 100$ rules out the prime $83$, so in particular $7$ is a bound on the largest nonsurjective prime for $C$. Noting that the expression on the right hand side of Equation~\ref{eq:B_bound} grows with $q$, we can apply Theorem \ref{T:Bbound} with $q = 7$ to obtain the value $B = 3.578 \times 10^{23}$ for which $\LikelyNonsurjectivePrimes(C; B)$ coincides with the set of nonsurjective primes associated with $C$. With this value of $B$, our implementation of the algorithm was still running after $24$ hours, after which we terminated it. Even if the version of Theorem \ref{T:Bbound} that relies on AHC could be made explicit, the value of $q^{11} \log^2(q N_A)$ in this example is on the order of $10^{11}$, which would still be a daunting prospect.

To execute the combined algorithm on all typical genus $2$ curves in the LMFDB - which at the time of writing constitutes $63{,}107$ curves - we have decided to take a fixed value of $B = 1000$ in Algorithm~\ref{A:combinedPart2algo}. The combined algorithm then takes about 4 hours on MIT's Lovelace computer, a machine with $2$ AMD EPYC 7713 2GHz processors, each with $64$ cores, and a total of 2TB of memory. The result of this computation of nonsurjective primes for these curves is available to view on the homepage of each curve in the LMFDB. In addition, the combined algorithm has been run on a much larger set of $1{,}743{,}737$ curves provided to us by Andrew Sutherland. See Section~\ref{S:examples} for the results of this computation. The largest Frobenius witness required for the smaller LMFDB dataset was $89$, and for Sutherland's larger dataset was $863$, so we chose $B = 1000$ to have our implementation work for both datasets.

\begin{rem}
It would be interesting to know if there is a uniform upper bound on the largest prime $\ell$ that could occur as a nonsurjective prime for the Jacobian of a typical genus $2$ curve defined over $\qq$, analogous to the conjectural bound of $37$ for the largest nonsurjective prime for elliptic curves defined over $\qq$ (see e.g. \cite[Introduction]{BPR}). Such a bound (if it exists) must be at least $31$, as shown by example \eqref{equation:nonsurj_31} in Section~\ref{S:examples}.
\end{rem}

Algorithm~\ref{A:combinedPart2algo} samples the characteristic polynomial of Frobenius $P_p(t)$ for each prime $p$ of good reduction for the curve up to a particular bound and applies Tests \ref{T:exc} and \ref{T:nonexc} to $P_p(t)$. Assuming that $\rho_{A,\ell}$ is surjective, we expect that the outcome of these tests should be independent for sufficiently large primes. More precisely, 

\begin{thm}\label{T:Bprob}
Let $C/\qq$ be a typical genus $2$ curve with Jacobian $A$ and suppose $\ell$ is an odd prime such that $\rho_{A,\ell}$ is surjective. There is an effective bound $B_0$ such that for any $B > B_0$, if we sample the characteristic polynomials $P_p(t)$ of Frobenius for $n$ primes $p \in [B,2B]$ chosen uniformly and independently at random, the probability that none of these pass Tests \ref{T:exc} or \ref{T:nonexc} is less than $3 \cdot \left(\frac{9}{10}\right)^n$.
\end{thm}

\begin{rem}
In fact, for each prime $\ell$ satisfying the conditions of Theorem \ref{T:Bprob}, there is an explicit constant $c_\ell \le \frac{9}{10}$ tending to $\frac{3}{4}$ as $\ell \rightarrow \infty$ which may be computed using Corollary \ref{C:ntrials} such that the bound $3 \cdot \left(\frac{9}{10}\right)^n$ in Theorem \ref{T:Bprob} can be replaced by $3 \cdot c_\ell^n$. While Theorem \ref{T:Bprob} allows us to consider these probabilities in theory, in practice the bound $B_0$ also arises from applying the Effective Chebotarev density theorem, so is therefore at least as large as as the bound from Theorem \ref{T:Bbound}, and is therefore similarly infeasible.
\end{rem}

The combined algorithm to probabilistically determine the nonsurjective primes of a typical genus $2$ curve over $\qq$ has been implemented in Sage \cite{sagemath}, and it will appear in a future release of this software\footnote{see \url{https://github.com/sagemath/sage/issues/30837} for the ticket tracking this integration.}. Until then, the implementation is available at the following repository:
\begin{center}

\url{https://github.com/ivogt161/abeliansurfaces}
\end{center}
The \path{README.md} file contains detailed instructions on its use. This repository also contains other scripts in both Sage and Magma \cite{magma} useful for verifying some of the results of this work; any filenames used in the sequel will refer to the above repository.

\subsection*{Outline of this paper}
  In Section~\ref{S:preliminaries},  
  we begin by reviewing the properties of the characteristic polynomials of Frobenius with a view towards computational aspects. 
   We also recall the classification of maximal subgroups of $\GSp_4(\ff_\ell)$. In Section~\ref{S:finite_list}, we explain Algorithm~\ref{A:combinedPart1algo} and establish Theorem \ref{T:mainthm}\eqref{part1}; that is, for each of the maximal subgroups of $\GSp_4(\ff_\ell)$ listed in Section~\ref{S:classification}, we generate a list of primes that provably contains all primes $\ell$ for which  the mod $\ell$ image of Galois is contained in this maximal subgroup. Theorem~\ref{T:mainbound} is proved in subsection~\ref{ssec:mainbound}. 
 In Section~\ref{S:test_ell}, we first prove a group-theoretic criterion (Proposition~\ref{P:grpthy}) for a subgroup of $\GSp_4(\ff_\ell)$ to equal $\GSp_4(\ff_\ell)$. Then, for each $\ell$ in the finite list from Section~\ref{S:finite_list}, we ascertain whether the characteristic polynomials of the Frobenius elements sampled satisfy the group-theoretic criterion; Theorem \ref{T:mainthm}\eqref{part2} and Theorem~\ref{T:Bbound} also follow from this study.
In Section~\ref{S:confidence} we prove Theorem~\ref{T:Bprob} concerning the probability of output error, assuming that Frobenius elements are uniformly distributed in $\rho_{A,\ell}(G_\qq)$. Finally, in Section~\ref{S:examples}, we close with remarks concerning the execution of the algorithm on the large dataset of genus $2$ curves mentioned above, and highlight some interesting examples that arose therein.

\ack{
This work was started at a workshop held remotely `at' the Institute for Computational and Experimental Research in Mathematics (ICERM) in Providence, RI, in May 2020, and was supported by a grant from the Simons Foundation (546235) for the collaboration `Arithmetic Geometry, Number Theory, and Computation'. It has also been supported by the National Science Foundation under Grant No. DMS-1929284 while the authors were in residence at ICERM during a Collaborate@ICERM project held in May 2022. I.V.'s work was partially supported by NSF grants DMS-1902743 and DMS-2200655.

We are grateful to Noam Elkies for providing interesting examples of genus $2$ curves in the literature, Davide Lombardo for helpful discussions related to computing geometric endomorphism rings, and to Andrew Sutherland for providing a dataset of Hecke characteristic polynomials that were used for executing our algorithm on all typical genus $2$ curves in the LMFDB, as well as making available the larger dataset of approximately $2$ million curves that we ran our algorithm on.
}

\section{Preliminaries}\label{S:preliminaries}

\subsection{Notation}\label{S:Notation}
Let $A$ be an abelian variety of dimension $g$ defined over $\qq$. Associated to $A$ is a positive integer $N = N_A$ called the \defi{conductor} (see e.g. \cite[Section 2]{brumer_kramer_1994}). We write $N_{\sq}$ for the largest integer such that $N_{\sq}^2 \mid N$.

Let $\ell$ be a prime. The $\ell$-adic Tate module \(T_\ell A \simeq \varprojlim A[\ell^n]\) of $A$ is a free $\zz_\ell$-module of rank $2g$.
For each prime $p$, we write $\Frob_p \in \Gal(\overline{\qq}/\qq)$ for an absolute Frobenius element associated to $p$. By a \defi{good prime} $p$ for an abelian variety $A$, we mean a prime $p$ for which $A$ has good reduction, or equivalently $p \nmid N_A$. If $p$ is a good prime for $A$, then the trace $a_p$ of the action of $\Frob_p$ on $T_\ell A$ is an integer. See Section~\ref{S:intcharpoly} for a discussion of the characteristic polynomial of Frobenius.

An abelian variety $A$ with geometric endomorphism ring $\mathbb{Z}$ is called \defi{typical}. A typical genus $2$ curve is a nice curve whose Jacobian is a typical abelian surface.

Let $V$ be a $4$-dimensional vector space over $\ff_\ell$ endowed with a nondegenerate skew-symmetric bilinear form $\langle\cdot,\cdot\rangle$. 
A subspace \(W \subseteq V\) is called \defi{isotropic} (for \(\langle\cdot,\cdot\rangle\)) if \(\langle w_1, w_2\rangle = 0\) for all \(w_1, w_2 \in W\).  A subspace \(W \subseteq V\) is called \defi{nondegenerate} (for \(\langle\cdot,\cdot\rangle\)) if \(\langle \cdot, \cdot \rangle\) restricts to a nondegenerate form on \(W\).
The \defi{general symplectic group} of $(V,\langle\cdot,\cdot\rangle)$ is the subgroup of \(\operatorname{GL}(V)\) defined by
\[ \GSp(V,\langle\cdot,\cdot\rangle) \colonequals \{M : \exists \  \simi(M) \in \mathbb{F}_\ell^\times : \langle Mv, Mw \rangle = \simi(M) \langle v,w \rangle \ \forall \ v,w \in V\}. \]

The map $M \mapsto \simi(M)$ is a surjective homomorphism from $\GSp(V,\langle\cdot,\cdot\rangle)$ to $\ff_\ell^\times$ called the \defi{similitude character}; its kernel is the \defi{symplectic group}, denoted $\Sp(V,\langle\cdot,\cdot\rangle)$. Usually the bilinear form is understood from context, and we drop $\langle\cdot,\cdot\rangle$ from the notation.

By a \defi{subquotient} $W$ of a Galois module $U$, we mean a Galois module $W$ that admits a surjection $U' \twoheadrightarrow W$ from a subrepresentation $U'$ of $U$. 

Since we are chiefly concerned with the sets $\LikelyNonsurjectivePrimes(C; B)$ and $\PossiblyNonsurjectivePrimes(C)$ for a fixed curve $C$, we will henceforth, for ease of notation, drop the $C$ from the notation for these sets.

\subsection{Integral characteristic polynomial of Frobenius}\label{S:intcharpoly} The theoretical result underlying the whole approach is the following.

\begin{thm}[{Weil, see \cite[Theorem 3]{SerreTate}}] \label{integrality}
Let $A$ be an abelian variety of dimension $g$ defined over $\qq$ and let $p$ be a prime of good reduction for $A$. There exists a monic integral polynomial $P_p(t) \in \zz[t]$ of degree $2g$ with constant coefficient $p^g$ such that for any $\ell \neq p$, the polynomial $P_p(t)$ is the characteristic polynomial of the action of $\Frob_p$ on $T_\ell A$.  Every root of \(P_p(t)\) has complex absolute value \(p^{1/2}\).
\end{thm}

 The polynomials $P_p(t)$ are computationally accessible by counting points on $C$ over $\ff_{p^r}$, $r = 1,2$. See \cite[Chapter~7]{poonen} for more details. In fact, $P_p(t)$ can be accessed via the \path{frobenius_polynomial} command in Sage. In particular, we denote the trace of Frobenius by $a_p$. 
By the Grothendieck-Lefschetz trace formula, if $A = \Jac{C}$, $p$ is a prime of good reduction for $C$, and $\lambda_1,\ldots,\lambda_{2g}$ are the roots of $P_p(t)$, then \( \# C(\ff_{p^r}) = p^r+1-\sum_{i=1}^{2g} \lambda_i^r. \)

\subsection{The Weil pairing and consequences on the characteristic polynomial of Frobenius}\label{S:WeilPairing}
The nondegenerate Weil pairing gives an isomorphism of Galois modules:
\begin{equation}\label{E:WeilpairingSim} T_\ell A \simeq \left(T_\ell A\right)^\vee \otimes_{\zz_\ell} \zz_\ell(1).\end{equation}
The Galois character acting on $\zz_\ell(1)$ is the \defi{$\ell$-adic cyclotomic character}, which we denote by $\cyc_\ell$.  
The integral characteristic polynomial for the action of $\Frob_p$ on $\zz_\ell(1)$ is simply $t-p$.  The integral characteristic polynomial for the action of $\Frob_p$ on $\left(T_\ell A\right)^\vee$ is the reversed polynomial
\[P^\vee_p(t) = P_p(1/t)\cdot t^{2g}/p^g \]
whose roots are the inverses of the roots of $P_p(t)$.

We now record a few easily verifiable consequences 
of the nondegeneracy of the Weil pairing when $\dim(A) = 2$.

\begin{lem}\label{L:charpolygen}\hfill
\begin{enumerate}[\upshape (i)]
\item The roots of $P_p(t)$ come in pairs that multiply out to $p$. In particular, $P_p(t)$ has no root with multiplicity $3$.
\item\label{L:charpolyshape} $P_p(t)=t^4-a_pt^3+b_pt^2-pa_pt+p^2$ for some $a_p,b_p\in \zz$.

\item\label{L:tracezero_charpolygen} 
If the trace of an element of $\GSpFL$ is $0 \mod \ell$, then its characteristic polynomial is reducible modulo $\ell$. In particular, this applies to $P_p(t)$ when \(a_p \equiv 0 \bmod{\ell}\).
\item\label{easy_necessary} If $A[\ell]$ is a reducible $G_{\mathbb{Q}}$-module, then $P_p(t)$ is reducible modulo $\ell$.
\end{enumerate}
\end{lem}
\begin{proof}
Parts (i) and (ii) are immediate from the fact that the non-degenerate Weil pairing allows us to pair up the four roots of $P_p(t)$ into two pairs that each multiply out to $p$. 

For part (iii), suppose that \(M \in \GSpFL\) has $\tr(M) = 0$. Then the characteristic polynomial $P_M(t)$ of $M$ is of the form $t^4+bt^2+c^2$. When the discriminant of $P_M$ is $0$ modulo $\ell$, the polynomial $P_M$ has repeated roots and is hence reducible. 
So assume that the discriminant of $P_M$ is nonzero modulo $\ell$. When $\ell \neq 2$, the result follows from \cite[Theorem~1]{carlitz}. 
When $\ell = 2$, a direct computation shows that the characteristic polynomial of a trace $0$ element of $\GSp_4(\ff_2)$ is either $(t+1)^4$ or $(t^2+t+1)^2$, which are both reducible.

Part (iv) is immediate from Theorem \ref{integrality} since $P_p(t) \mod \ell$ by definition is the characteristic polynomial for the action of $\Frob{p}$ on $A[\ell]$. 
\end{proof}

\subsection{\boldmath Maximal subgroups of \texorpdfstring{$\GSp_4(\ff_\ell)$}{Gsp4Fell}} \label{S:classification} 

Mitchell \cite{Mitchell} classified the maximal subgroups of $\PSp_4(\ff_\ell)$ in 1914. This can be used to deduce the following classification of maximal subgroups of \(\GSp_4(\ff_\ell)\) with surjective similitude character.

\begin{lem}[Mitchell] \label{lem:max_subgroups} 
Let $V$ be a $4$-dimensional $\ff_\ell$-vector space endowed with a nondegenerate skew-symmetric bilinear form $\omega$. Then any proper subgroup $G$ of $\operatorname{GSp}(V,\omega)$ with surjective similitude character is contained in one of the following types of maximal subgroups. 
\begin{enumerate}[\upshape (1)]
    \item\label{max:red} \textbf{Reducible maximal subgroups} \hfill
    
    \begin{enumerate}[\upshape (a)]
        \item\label{max:1+3} Stabilizer of a 1-dimensional isotropic subspace for $\omega$.
        \item\label{max:2+2} Stabilizer of a 2-dimensional isotropic subspace for $\omega$.
    \end{enumerate}
    
   \item\label{max:irred} \textbf{Irreducible subgroups governed by a quadratic character}\hfill

    Normalizer $G_\ell$ of the group \(M_\ell\) that preserves each summand in a direct sum decomposition \(V_1 \oplus V_2\) of $V$ into two $2$-dimensional subspaces, where \(V_1\) and \(V_2\) are jointly defined over \(\ff_\ell\) and either:

    \begin{enumerate}[\upshape (a)]
    \item\label{max:congruence} both nondegenerate for $\omega$; or
    \item\label{max:quadric} both isotropic for $\omega$.
     \end{enumerate}
     Moreover, $M_\ell$ is an index $2$ subgroup of $G_\ell$.

\item\label{max:twistedcubic} \textbf{Stabilizer of a twisted cubic}

$\GL(W)$ acting on \(\Sym^3W \simeq V\), where \(W\) is a dimension \(2\) \(\ff_\ell\)-vector space.

\item\label{max:exceptional} \textbf{Exceptional subgroups} See Table~\ref{tab:exceptional} for explicit generators for the groups described below.

    \begin{enumerate}[\upshape (a)]
        \item\label{max:G1920} When $\ell \equiv \pm 3 \bmod{8}$: group with image \(G_{1920}\) in \(\PGSp(V, \omega)\) of order $1920$.
        \item\label{max:G720} When $\ell \equiv \pm 5 \bmod{12}$ and \(\ell \neq 7\): group with image \(G_{720}\) in \(\PGSp(V, \omega)\) of order $720$.
        \item\label{max:G5040} When \(\ell = 7\): group with image \(G_{5040}\) in \(\PGSp(V, \omega)\) of order $5040$.
    \end{enumerate}
    \end{enumerate}
\end{lem}
\noindent See Lemma~\ref{L:recognizingmaxl} for a detailed description of the groups $G_\ell$ and $M_\ell$ in case ~\eqref{max:irred} above.
\begin{rem}\label{R:labelling}
We have chosen to label the maximal subgroups in the classification using invariant subspaces for the symplectic pairing $\omega$ on $V$, following the more modern account due to Aschbacher (see \cite[Section~3.1]{lombardo2016explicit}; for a more comprehensive treatment see \cite{kleidman_liebeck_1990}). For the convenience of the reader, we record the correspondence between Mitchell's original labels and ours below.

\begin{table}[h]\label{T:maxsubgrouplabels}
\begin{tabular}{|c|c|}
\hline
\textbf{Mitchell's label} & \textbf{Label in Lemma~\ref{lem:max_subgroups}} \\
\hline
Group having an invariant point and plane & \ref{max:1+3} \\ \hline
Group having an invariant parabolic congruence & \ref{max:2+2} \\ \hline
Group having an invariant hyperbolic or elliptic congruence & \ref{max:congruence} \\ \hline
Group having an invariant quadric & \ref{max:quadric} \\ \hline
\end{tabular}
\vspace{3pt}
\caption{Dictionary between maximal subgroup labels in \cite{Dieulefait}/\cite{Mitchell} and Lemma~\ref{lem:max_subgroups}}
\end{table}
\end{rem}

\begin{rem}
 The maximal subgroups in \eqref{max:red} are the analogues of the Borel subgroup of $\GL_2(\ff_\ell)$. The maximal subgroups in \eqref{max:irred} when the two subspaces $V_1,V_2$ in the direct sum decomposition are individually defined over $\ff_{\ell}$ are the analogues of normalizers of the split Cartan subgroup of $\GL_2(\ff_\ell)$. When the two subspaces $V_1,V_2$ are not individually defined over $\ff_\ell$ instead, the maximal subgroups in \eqref{max:irred} are analogues of the normalizers of the non-split Cartan subgroups of $\GL_2(\ff_\ell)$. 
 \end{rem}
 
 \begin{rem}
 We briefly explain why the action of \(\GL_2(\ff_\ell)\) on \(\Sym^3(\ff_\ell^2)\) preserves a nondegenerate symplectic form.  It suffices to show that the restriction to \(\SL_2(\ff_\ell)\) fixes a vector in \(\bigwedge^2\Sym^3(\ff_\ell^2)\).  
 If \(V\) has basis \(\{x,y\}\), then an invariant vector is \(x^3 \wedge y^3 - 3x^2y \wedge xy^2\).
 \end{rem}

\begin{rem}
One can extract explicit generators of the exceptional maximal subgroups from Mitchell's original work\footnote{Mitchell's notation for \(\PGSp_4(\ff_\ell)\) is \(A_\nu(\ell)\) and for \(\PSp_4(\ff_\ell)\) is \(A_1(\ell)\).}. Indeed \cite[the proof of Theorem 8, page~390]{Mitchell} gives four explicit matrices that generate a \(G_{1920}\) (which is unique up to conjugacy in \(\PGSp_4(\ff_\ell)\)). Mitchell's description of the other exceptional groups is in terms of certain projective linear transformations called skew perspectivities attached to a direct sum decomposition \(V = V_1 \oplus V_2\) into \(2\)-dimensional subspaces.  A \defi{skew perspectivity of order \(n\) with axes \(V_1\) and \(V_2\)} is the projective linear transformation that scales \(V_1\) by a primitive \(n\)th root of unity and fixes \(V_2\). This proof also gives the axes of the skew perspectivities of order \(2\) and \(3\) that generate the remaining exceptional groups \cite[pages 390-391]{Mitchell}.  Table~\ref{tab:exceptional}
lists generators of (one representative of the conjugacy class of) each of the exceptional maximal subgroup extracted from Mitchell's descriptions. 
 
In the file \texttt{exceptional.m} publicly available with our code, we  verify that Magma's list of conjugacy classes of maximal subgroups of $\GSp_4(\ff_\ell)$ agree with those described in Lemma~\ref{lem:max_subgroups} for \(3 \leq \ell \leq 47\).
\end{rem}

\begin{rem}
The classification of exceptional maximal subgroups of \(\PSp_4(\ff_\ell)\) is more subtle than that of \(\PGSp_4(\ff_\ell)\), because of the constraint on the similitude character of matrices in \(\PSp_4(\ff_\ell)\).  While the similitude character is not well-defined on \(\PGSp_4(\ff_\ell)\) (multiplication by a scalar \(c \in \ff_\ell^\times\) scales the similitude character by \(c^2\)) it is well-defined modulo squares.  The group \(\PSp_4(\ff_\ell)\) is the kernel of this natural map:

\[ 1 \rightarrow \PSp_4(\ff_\ell) \rightarrow \PGSp_4(\ff_{\ell}) \xrightarrow{\simi} \ff_\ell^\times /(\ff_\ell^\times)^2 \simeq \{\pm 1\} \rightarrow 1. \]

An exceptional subgroup of \(\PGSp_4(\ff_\ell)\) gives rise to an exceptional subgroup of \(\PSp_4(\ff_\ell)\) of either the same size or half the size depending on the image of \(\mult\) restricted to that subgroup, which in turn depends on the congruence class of \(\ell\).  For this reason, the maximal exceptional subgroups of $\PSp_4(\ff_\ell)$ in Mitchell's original classification (also recalled in Dieulefait \cite[Section 2.1]{Dieulefait}) can have order \(1920\) \textit{or} \(960\) and \(720\) \textit{or} \(360\) depending on the congruence class of \(\ell\), and \(2520\) (for \(\ell = 7\)).  Such an exceptional subgroup gives rise to a \textit{maximal} exceptional subgroup of \(\PGSp_4(\ff_\ell)\) only when \(\mult\) is surjective (i.e., its intersection with \(\PSp_4(\ff_\ell)\) has index two), which explains the restricted congruence classes of \(\ell\) for which they arise.
\end{rem}

We now record a lemma that directly follows from the structure of maximal subgroups described above. This lemma will be used in Section~\ref{S:test_ell} to devise a criterion for a subgroup of $\GSp_4(\ff_\ell)$ to be the entire group. For an element $T$ in $ \GSp_4(\ff_\ell)$, let $\tr(T)$, $\midd(T)$, $\simi(T)$ denote the trace of $T$, the middle coefficient of the characteristic polynomial of $T$, and the similitude character applied to $T$ respectively\footnote{Explicitly, the characteristic polynomial of \(T\) is therefore \(t^4 - \tr(T)t^3 + \midd(T)t^2 - \simi(T)\tr(T)t + \simi(T)^2\).}. For a scalar \(\lambda\), we have 
\[ \tr(\lambda T) = \lambda\tr(T), \quad \midd(\lambda T) = \lambda^2\midd(T), \quad \simi(\lambda T) = \lambda^2\simi(T).\]  
Hence the quantities \(\tr(T)^2/\simi(T)\) and \(\midd(T)/\simi(T)\) are well-defined on \(\PGSp_4(\ff_\ell)\).
For $\ell > 2$ and $* \in \{720,1920, 5040\}$, define
\begin{equation}\label{E:precomplist} C_{\ell,*} \colonequals \left\{ \left.\left(\frac{\tr(T)^2}{\simi(T)},\frac{\midd(T)}{\simi(T)}\right) \ \right| \ T \in \textup{ exceptional subgroup of projective order $*$}  \right\} \end{equation}

\begin{lem}\label{L:conseqclass}\hfill
\begin{enumerate}[\upshape (1)]
\item\label{L:easyconseq} In cases~\ref{max:congruence} and \ref{max:quadric} of Lemma \ref{lem:max_subgroups}:
\begin{enumerate}[\upshape (a) ] 
\item\label{L:tracezero_GSp4poly} every element in $G_\ell \setminus M_\ell$ has trace $0$, and,
\item\label{L:parsubgp} the group $M_\ell$ stabilizes a non-trivial linear subspace of $\overline{\ff}_\ell^4$.
\end{enumerate}
\item\label{L:twistedcubic}
 Every element that is contained in a maximal subgroup corresponding to the stabilizer of a twisted cubic has a reducible characteristic polynomial.
\item\label{L:precompjust} For $* \in \{1920,720\}$, the set $C_{\ell,*}$ defined in ~\eqref{E:precomplist} equals the reduction modulo $\ell$ of the elements of the set $C_{*}$ below. 
{\small
\begin{align*} C_{1920}  &= \{(0, -2),(0, -1),(0, 0),(0, 1),(0, 2),(1, 1),(2, 1),(2, 2),(4, 2),(4, 3),(8, 4),(16, 6)\} \\
C_{720} &= \{ (0, 1), (0, 0), (4, 3), (1, 1), (16, 6), (0, 2), (1, 0), (3, 2), (0, -2) \}
\end{align*}
}
We also have
{\small
\[ C_{7,5040} = \{(0, 0), (0, 1), (0, 2), (0, 5), (0, 6), (1, 0), (1, 1), (2, 6), (3, 2), (4, 3), (5, 3), (6, 3) \}. \]
}
\end{enumerate}
\end{lem}
\begin{proof}\hfill
\begin{enumerate}
    \item In cases~\ref{max:congruence} and \ref{max:quadric} of Lemma~\ref{lem:max_subgroups}, since any element of the normalizer $G_\ell$ that is not in $M_{\ell}$ switches elements in the two subspaces $V_1$ and $V_2$ (i.e. maps elements in the subspace $V_1$ in the decomposition $V_1 \oplus V_2$ to elements in $V_2$ and vice-versa), it follows that any element in $G_\ell \setminus M_\ell$ has trace zero. 
    \item The conjugacy class of maximal subgroups corresponding to the stabilizer of a twisted cubic comes from the embedding $\GL_2(\ff_\ell) \xrightarrow{\iota} \operatorname{GSp}_4(\ff_\ell)$ induced by the natural action of $\GL_2(\ff_\ell)$ on the space of monomials of degree $3$ in $2$ variables. If $M$ is a matrix in $\GL_2(\ff_\ell)$ with eigenvalues $\lambda,\mu$ (possibly repeated), then the eigenvalues of $\iota(M)$ are $\lambda^3,\mu^3,\lambda^2\mu,\lambda\mu^2$ and hence the characteristic polynomial of $\iota(M)$ factors as $(T^2-(\lambda^3+\mu^3)T+\lambda^3\mu^3) (T^2-(\lambda^2 \mu+\lambda\mu^2)T+\lambda^3\mu^3)$ over $\ff_\ell$ which is reducible over $\ff_\ell$.
    \item This follows from the description of the maximal subgroups given  in Table~\ref{tab:exceptional}. 
    Each case (except \(G_{5040}\) that only occurs for \(\ell=7\)) depends on a choice of a root of a quadratic polynomial.  In the associated file \texttt{exceptional\_statistics.sage}, we generate the corresponding finite subgroups over the appropriate quadratic number field to compute \(C_*\).
    It follows that the corresponding values for the subgroup $G_{*}$ in $\GSp_4(\ff_\ell)$ can be obtained by reducing these values modulo $\ell$. 
    Since the group $G_{5040}$ only appears for $\ell=7$, we directly compute the set $C_{7,5040}$. \qedhere
\end{enumerate}
\end{proof}

\begin{rem}
The condition in Lemma~\ref{L:conseqclass}\eqref{L:precompjust} is the analogue of the condition \cite[Proposition 19 (iii)]{Serre72} to rule out exceptional maximal subgroups of \(\GL_2(\ff_\ell)\).
\end{rem}
 
We end this subsection by including the following lemma, to further highlight the similarities between the above classification of maximal subgroups of $\GSp_4(\ff_\ell)$ and the more familiar classification of maximal subgroups of $\GL_2(\ff_{\ell})$. This lemma is not used elsewhere in the article and is thus for expositional purposes only.
 \begin{lem}\label{L:recognizingmaxl} \hfill
  \begin{enumerate}[\upshape (1)]
   \item\label{L:Cfq} The subgroup $M_{\ell}$ in the case \eqref{max:congruence} when the two nondegenerate subspaces $V_1$ and $V_2$ \textbf{are} individually defined over $\ff_{\ell}$ is isomorphic to
   \[ \{ (m_1,m_2) \in \GL_2(\ff_\ell)^2 \ | \ \det(m_1) = \det(m_2) \}. \]
   In particular, the order of $M_{\ell}$ is $\ell^2(\ell-1)(\ell^2-1)^2$. 
   
   \item\label{L:Qfq} The subgroup $M_{\ell}$ in the case \eqref{max:quadric} when the two isotropic subspaces $V_1$ and $V_2$ \textbf{are} individually defined over $\ff_{\ell}$ is isomorphic to 
   \[ \{ (m_1,m_2) \in \GL_2(\ff_\ell)^2 \ | \ m_1^Tm_2 = \lambda I, \textup{ for some } \lambda \in \ff_{\ell}^* \}. \]
   In particular, the order of $M_{\ell}$ is $\ell(\ell-1)^2(\ell^2-1)$.
   
   \item\label{L:Cfq2} The subgroup $M_{\ell}$ in the case \eqref{max:congruence} when the two nondegenerate subspaces $V_1$ and $V_2$ \textbf{are not} individually defined over $\ff_{\ell}$ is isomorphic to 
   \[ \{ m \in \GL_2(\ff_{\ell^2}) \ | \ \det(m) \in \ff_\ell^* \}. \]
   In particular, the order of $M_{\ell}$ is $\ell^2(\ell-1)(\ell^4-1)$.
   
   \item\label{L:Qfq2} The subgroup $M_{\ell}$ in the case \eqref{max:quadric} when the two isotropic subspaces $V_1$ and $V_2$ \textbf{are not} individually defined over $\ff_{\ell}$ is isomorphic to $\GU_2(\ff_{\ell^2})$, i.e., 
   \[ \{ m \in \GL_2(\ff_{\ell^2}) \ | \ m^T \iota(m) = \lambda I, \textup{ for some } \lambda \in \ff_{\ell}^* \}, \]
   where $\iota$ denotes the natural extension of the Galois automorphism of $\ff_{\ell^2}/\ff_\ell$ to $\GL_2(\ff_{\ell^2})$.
   In particular, the order of $M_{\ell}$ is $\ell(\ell^2-1)^2$.
   
  \end{enumerate}
 \end{lem}
\begin{proof}
Given a direct sum decomposition $V_1 \oplus V_2$ of a vector space $V$ over $\ff_q$, we get a natural embedding of $\Aut(V_1) \times \Aut(V_2)$ ($\cong \GL_2(\ff_q)^2$) into $\Aut(V)$ ($\cong \GL_4(\ff_q)$), whose image consists of automorphisms that preserve this direct sum decomposition. We will henceforth refer to elements of $\Aut(V_1) \times \Aut(V_2)$ as elements of $\Aut(V)$ using this embedding. To understand the subgroup $M_{\ell}$ of $\GSp_4(\ff_q)$ in cases~\eqref{L:Cfq} and \eqref{L:Qfq} where the two subspaces in the direct sum decomposition are individually defined over $\ff_q$, we need to further impose the condition that the automorphisms in the image of the map $\Aut(V_1) \times \Aut(V_2) \rightarrow \Aut(V)$ preserve the symplectic form $\omega$ on $V$ up to a scalar.

In \eqref{L:Cfq}, without any loss of generality, the two nondegenerate subspaces $V_1$ and $V_2$ can be chosen to be orthogonal complements under the nondegenerate pairing $\omega$, and so by Witt's theorem, in a suitable basis for $V_1 \oplus V_2$ obtained by concatenating a basis of $V_1$ and a basis of $V_2$, the nondegenerate symplectic pairing $\omega$ has the following block-diagonal shape:

 \[ J \colonequals
\begin{bmatrix}
0  & 1  & & \\
-1 & 0 & & \\
   &   & 0  & 1 \\
   &   & -1 & 0 \\
\end{bmatrix}.
 \]

 The condition that an element $(m_1,m_2) \in \Aut(V_1) \oplus \Aut(V_2)$ preserves the symplectic pairing up to a similitude factor of $\lambda$ is the condition $(m_1,m_2)^T J (m_1,m_2) = \lambda J$, which boils down to $\det(m_1) = \lambda = \det(m_2)$. 
 
 Similarly, in \eqref{L:Qfq}, without any loss of generality, by Witt's theorem, in a suitable basis for $V_1 \oplus V_2$ obtained by concatenating a basis of the isotropic subspace $V_1$ and a basis of the isotropic subspace $V_2$, the nondegenerate symplectic pairing $\omega$ has the following block-diagonal shape. 
 
 \[ J \colonequals
\begin{bmatrix}
& & 0  & 1 \\
& & 1 & 0 \\
0  & -1 & & \\
-1 & 0 & & \\
\end{bmatrix}.
 \]
 
The condition that an element $(m_1,m_2) \in \Aut(V_1) \oplus \Aut(V_2)$ preserves the symplectic pairing up to a similitude factor of $\lambda$ is the condition $(m_1,m_2)^T J (m_1,m_2) = \lambda J$, which again boils down to $m_1^T m_2 = \lambda I$.

If we have a subspace $W$ defined over $\ff_{q^2}$ but not defined over $\ff_q$, and we let $\overline{W}$ denote the conjugate subspace and further assume that $W \oplus \overline{W}$ gives a direct sum decomposition of $V$, then we get a natural embedding of $\Aut(W)$ ($\cong \GL_2(\ff_{q^2})$) into $\Aut(V)$ ($\cong \GL_4(\ff_q)$) whose image consists of automorphisms that commute with the natural involution of $V \otimes \ff_{q^2}$ induced by the Galois automorphism of $\ff_{q^2}$ over $\ff_q$. The proofs of cases ~\eqref{L:Cfq2} and ~\eqref{L:Qfq2} are analogous to the cases ~\eqref{L:Cfq} and ~\eqref{L:Qfq} respectively, by using the direct sum decomposition $W \oplus \overline{W}$ and letting $m_2 = \iota(m_1)$. The condition that $\det(m_1) = \det(m_2)$ in ~\eqref{L:Cfq} becomes the condition $\det(m_1) = \det(m_2) =  \det{\overline{m_1}} = \overline{\det(m_1)}$, or equivalently, that $\det(m_1) \in \ff_q$ in ~\eqref{L:Cfq2}. Similarly, the condition that $m_1^Tm_2 = \lambda I$ in ~\eqref{L:Qfq} becomes the condition that $m_1^T \iota(m_1) = \lambda I$ in ~\eqref{L:Qfq2}. \qedhere
\end{proof}

\subsection{Image of inertia and (tame) fundamental characters}\label{S:inertiaimage}

Dieulefait \cite{Dieulefait} used Mitchell's work described in the previous subsection to classify the maximal subgroups of $\GSp_4(\ff_\ell)$ that could occur as the image of $\rho_{A,\ell}$ . This was achieved via an application of a fundamental result of Serre and Raynaud that strongly constrains the action of inertia at $\ell$, and which we now recall.

Fix a prime $\ell > 3$ that does not divide the conductor $N$ of $A$. Let $I_\ell$ be an inertia subgroup at $\ell$. Let \(\psi_n \colon I_\ell \to \ff_{\ell^n}^\times\) denote a (tame) fundamental character of level \(n\).  The \(n\) Galois-conjugate fundamental characters \(\psi_{n,1}, \dots, \psi_{n,n}\) of level \(n\) are given by \(\psi_{n,i} \colonequals \psi_n^{\ell^i}\). Recall that the fundamental character of level \(1\) is simply the mod \(\ell\) cyclotomic character $\cyc_\ell$, and that the product of all fundamental characters of a given level is the cyclotomic character.

\begin{thm}[Serre \cite{Serre72}, Raynaud \cite{Raynaud74}, cf. Theorem 2.1 in \cite{Dieulefait}]\label{T:serre-raynaud}
Let \(\ell\) be a semistable prime for \(A\).
Let \(V/\ff_\ell\) be an \(n\)-dimensional Jordan--H\"older factor of the \(I_\ell\)-module \(A[\ell]\).  Then \(V\) admits a $1$-dimensional \(\ff_{\ell^n}\)-vector space structure such that \(\rho_{A, \ell}|_{I_\ell}\) acts on \(V\) via the character
\[ \psi_{n,1}^{d_1}\cdots\psi_{n,n}^{d_n}  \]
with each $d_i$ equal to either $0$ or $1$.
\end{thm}

On the other hand, the following fundamental result of Grothendieck constrains the action of inertia at semistable primes \(p \neq \ell\).

\begin{thm}[{Grothendieck \cite[Expos\'e IX, Prop 3.5]{SGA7-I}}]\label{T:grothendieck}
Let \(A\) be an abelian variety over a number field \(K\).  Then \(A\) has semistable reduction at \(\mathfrak{p} \) not above \(\ell\) if and only if the action of \(I_{\mathfrak{p}} \subset G_K\) on \(T_\ell A\) is $2$-step unipotent (i.e. $(\rho_{A,\ell}(g) - I)^2 = 0$ for all $g \in I_{\mathfrak{p}}$).
\end{thm}

Combining these two results allows one fine control of the determinant of a subquotient of $A[\ell]$; this will be used in Section~\ref{S:finite_list}.

\begin{cor}\label{C:subqtdet} 
Let $A/\qq$ be an abelian surface, and let \(X_\ell\) be a Jordan--H\"older factor of the \(\overline{\ff}_\ell[G_\qq]\)-module \(A[\ell]\otimes\overline{\ff}_\ell\).
If $\ell$ is a semistable prime, then 
\[\det X_\ell \simeq \epsilon \cdot \cyc_\ell^x\]
for some character $\epsilon \colon G_{\qq} \to \bar{\ff}_{\ell}$ that is unramified at $\ell$ and some $0 \leq x \leq \dim X_{\ell}$.
Moreover, \(\epsilon^{120} = 1.\)
\end{cor}

\begin{proof}
The first part follows immediately from Theorem~\ref{T:serre-raynaud}. 
For the fact that \(\epsilon^{120} = 1\), we will show that \(\epsilon^{120}\) is unramified everywhere; the result will then follow from the fact that there are no nontrivial unramified characters of \(G_\qq\). Since $\epsilon$ is unramifed at $\ell$, so too is \(\epsilon^{120}\), so it suffices to show that \(\epsilon^{120}\) is unramified at primes $p \neq \ell$. From \cite[Theorem 7.2]{larson_vaintrob_2014} we know that every abelian surface attains semistable reduction over an extension \(K/\qq\) with \([K:\qq]\) dividing \(120\); therefore by Theorem~\ref{T:grothendieck} we have that the action of $I_\mathfrak{p} \subset G_K$ on $T_\ell A$ is $2$-step unipotent for any prime $\mathfrak{p} \mid p$ of \(K\). Hence the action of the 120th power of any element of $I_p$ is unipotent, and thus has trivial determinant.  
\end{proof}

We can now state Dieulefait's classification of maximal subgroups of \(\GSp_4(\ff_\ell)\) that can occur as the image \(\rho_{A, \ell}(G_\qq)\) for a semistable prime \(\ell > 7\).

\begin{prop}[{\cite{Dieulefait}}]\label{prop:max_semistable}
Let \(A\) be the Jacobian of a genus \(2\) curve defined over \(\qq\) with Weil pairing \(\omega\) on \(A[\ell]\).  If \(\ell > 7\) is a semistable prime, then \(\rho_{A, \ell}(G_\qq)\) is either all of \(\GSp(A[\ell], \omega) \) or it is contained in one of the maximal subgroups of Types \eqref{max:red} or \eqref{max:irred} in Lemma~\ref{lem:max_subgroups}.
\end{prop}

See also \cite[Proposition~3.15]{lombardo2016explicit} for an expanded exposition of why the image of $G_\qq$ cannot be contained in maximal subgroup of Type~\eqref{max:twistedcubic} for a semistable prime \(\ell > 7\).

\begin{rem}\label{R:ExcBound}However, if $\ell$ is a prime of additive reduction, or if $\ell \leq 7$, then the image of $G_\qq$ may also be contained in any of the four types of maximal subgroups described in Lemma~\ref{lem:max_subgroups}. Nevertheless, by \cite[Theorem 6.6]{LombaroVerzobio}, for any prime \(\ell > 24\), we have that the exponent of the projective image is bounded
\(\exp(\pp \rho_{A, \ell}) \geq (\ell - 1)/12 \). Since \(\exp(G_{1920}) = 2 \exp(S_6) = 120\) and \(\exp(G_{720}) = \exp(S_5) = 60\), the exceptional maximal subgroups cannot occur as \(\rho_{A,\ell}(G_\qq)\) for \(\ell > 1441\).
\end{rem}

\subsection{A consequence of the Chebotarev density theorem}\label{S:Chebotarev} Let $K/\qq$ be a finite Galois extension with Galois group $G = \Gal(K/\qq)$ and absolute discriminant $d_K$. Let $S \subseteq G$ be a nonempty subset that is closed under conjugation. By the Chebotarev density theorem, we know that
\begin{equation} \label{E:CDT}
\lim_{x \to \infty} \frac{|\{p \leq x : p \text{ is unramified in } K \text{ and } \Frob_p \in S \}|}{|\{p \leq x\}|} = \frac{|S|}{|G|}.
\end{equation}
Let $p$ be the least prime such that $p$ is unramified in $K$ and $\Frob_p \in S$. There are effective versions of the Chebotarev density theorem that give bounds on  $p$.  The best known unconditional bounds are polynomials in $d_K$ \cite{MR553223, MR3986919,MR4472459}.
Under GRH, the best known bounds are polynomials in $\log d_K$. In particular Bach and Sorenson \cite{MR1355006} showed that under GRH, 
\begin{equation} \label{ECDT-GRH}
p \leq (4 \log d_K + 2.5 [K:\qq] + 5)^2.
\end{equation}

The present goal is to give an effective version of the Chebotarev density theorem in the context of abelian surfaces. We will use a corollary of \eqref{ECDT-GRH} that is noted in \cite{MaWa2020} which allows for the avoidance of a prescribed set of primes by taking a quadratic extension of $K$. We do this because we will take $K = \qq(A[\ell])$, and $p$ being unramified in $K$ is not sufficient to imply that $p$ is a prime of good reduction for $A$. Lastly, we will use that by \cite[Proposition 6]{serre81}, if $K/\qq$ is finite Galois, then
\begin{equation} \label{logdK-bound}
\log d_K \leq ([K:\qq] - 1)\log \rad(d_K) + [K : \qq] \log([K:\qq]),
\end{equation}
where $\rad n = \prod_{p \mid n} p$ denotes the radical of an integer $n$.

\begin{lem}\label{lem:effective_cheb} Let $A/\qq$ be a typical principally polarized abelian surface with conductor \(N_A\). 
Let \(q\) be a prime.
Let $S \subseteq \rho_{A,q}(G_\qq)$ be a nonempty subset that is closed under conjugation. Let $p$ be the least prime of good reduction for $A$ such that $p \neq q$ and $\rho_{A,q}(\Frob_p) \in S$. 
    Assuming GRH, we have 
    \[ p \leq \left( 4\left[ (2q^{11} - 1)\log \rad(2 q N_A) + 22 q^{11} \log(2q) \right] + 5q^{11} + 5  \right)^2.\]
\end{lem}
\begin{proof} Let $K = \qq(A[q])$. Then the image of $\rho_{A,q}$ is $\Gal(K/\qq)$, the extension $K/\qq$ is Galois and 
\begin{equation*} \label{bound-degK} [K:\qq] \leq |\GSp_4(\ff_q)| = q^4 (q^4-1)(q^2-1)(q-1)  \leq q^{11}. \end{equation*}
As $\rad(d_K)$ is the product of primes that ramify in $\qq(A[q])$, the criterion of N\'eron-Ogg-Shafarevich for abelian varieties \cite[Theorem 1]{SerreTate} implies that $\rad (d_K)$ divides $\rad(q N_A)$. Let $\tilde{K} \colonequals K(\sqrt{m})$ where $m \colonequals \rad(2N_A)$. Note that the primes that ramify in $\tilde{K}$ are precisely $2$, $q$, and the primes of bad reduction for $A$. Thus  $\rad(d_{\tilde{K}}) = \rad(2 q N_A)$. Moreover $[\tilde{K}:\qq] \leq 2 q^{11}$ and by  \eqref{logdK-bound},
\begin{equation*} \label{bound-disc} \log(d_{\tilde{K}}) \leq (2q^{11} - 1) \log\rad(2qN_A) + 22 q^{11} \log(2q). \end{equation*}
Applying \cite[Corollary 6]{MaWa2020} to the field $\tilde{K}$, we get that (under GRH) there exists a prime $p$ satisfying the claimed bound, that does not divide $m$, and for which $\rho_{A,q}(\Frob_p) \in S$.
\end{proof}

\section{Finding a finite set containing all nonsurjective primes}\label{S:finite_list}

In this section we describe Algorithm~\ref{A:combinedPart1algo} referenced in Theorem \ref{T:mainthm}\eqref{part1}, which is a reformulation of the algorithm of Dieulefait \cite{Dieulefait} with some modest improvements.  This algorithm produces a finite list $\PossiblyNonsurjectivePrimes$ that provably includes all nonsurjective primes $\ell$. We also prove Theorem~\ref{T:mainbound}.

Since our goal is to produce a finite list (from which we will later remove extraneous primes) it is harmless to include the finitely many bad primes as well as \(2, 3, 5, 7\).
Using Proposition~\ref{prop:max_semistable},
it suffices to find conditions on \(\ell>7\) for which \(\rho_{A, \ell}(G_\qq)\) could be contained in one of the maximal subgroups of type \eqref{max:red} and \eqref{max:irred} in Lemma~\ref{lem:max_subgroups}.  We first find primes $\ell$ for which $\rho_{A,\ell}$ has (geometrically) reducible image (and hence is contained in a maximal subgroup in case \eqref{max:red} of Lemma~\ref{lem:max_subgroups} or in a subgroup \(M_\ell\) in case \eqref{max:irred}). 
To treat the geometrically irreducible cases,
we then make use of the observation from Lemma~\ref{L:conseqclass}~\ref{L:tracezero_GSp4poly} that every element outside of an index \(2\) subgroup has trace \(0\).

\begin{algo}\label{A:combinedPart1algo} Given a typical genus \(2\) curve $C/\qq$  with conductor \(N\) and Jacobian \(A\), compute a finite list \(\PossiblyNonsurjectivePrimes\) of primes as follows.
\begin{enumerate}[\upshape (1)]
    \item Initialize \(\PossiblyNonsurjectivePrimes = [2, 3, 5, 7]\).
    \item\label{step:bad} Add to \(\PossiblyNonsurjectivePrimes\) all primes dividing \(N\).
    \item\label{step:non_abs_irred} Add to \(\PossiblyNonsurjectivePrimes\) the good primes $\ell$
    for which \(\rho_{A, \ell} \otimes \bar{\ff}_{\ell}\) could be reducible via Algorithms \ref{alg:odd}, \ref{alg:related}, and \ref{alg:cuspforms}.
    \item\label{step:quad_char} Add to \(\PossiblyNonsurjectivePrimes\) the good primes $\ell$ for which \(\rho_{A, \ell} \otimes \bar{\ff}_{\ell}\) could be irreducible but nonsurjective via Algorithm \ref{alg:quad_char}.
    \item Return \(\PossiblyNonsurjectivePrimes\).
\end{enumerate}
\end{algo}

At a very high-level, each of the subalgorithms of Algorithm~\ref{A:combinedPart1algo} makes use of a set of \defi{auxiliary good primes} \(p\). We compute the integral characteristic polynomial of Frobenius \(P_p(t)\) and use it to constrain those \(\ell \neq p\) for which the image could have a particular shape.

\begin{rem}
Even though robust methods to compute the conductor $N$ of a genus \(2\) curve are not implemented at the time of writing, the odd-part $N_{ \text{odd}}$ of $N$ can be computed via \href{https://fossies.org/linux/pari/src/functions/elliptic_curves/genus2red}{\texttt{genus2red}} function of PARI and the \href{https://doc.sagemath.org/html/en/reference/arithmetic_curves/sage/interfaces/genus2reduction.html}{\texttt{genus2reduction}} module of SageMath, both based on an algorithm of Liu \cite{liu1994conducteur}. Moreover, \cite[Theorem 6.2]{brumer_kramer_1994} bounds the $2$-exponent of $N$ above by $20$ and hence $N$ can be bounded above by $2^{20} N_{ \text{odd}}$.  While these algorithms can be run only with the bound \(2^{20} N_{\text{odd}}\), doing so substantially increases the run-time of the limiting Algorithm~\ref{alg:cuspforms}.
\end{rem}

We now explain each of these steps in detail.

\subsection{Good primes that are not geometrically irreducible}\label{S:absredprimes}
In this section we describe the conditions that $\ell$ must satisfy for the base-extension
\(\oAl \colonequals A[\ell] \otimes_{\ff_\ell}\bar{\ff}_\ell\)
to be reducible.
In this case, the representation $\oAl$ is an extension
\begin{equation}\label{red_extn} 0 \to X_\ell \to \oAl \to Y_\ell \to 0\end{equation}
of a (quotient) representation $Y_\ell$ by a (sub) representation $X_\ell$.  
Recall that $N_{\sq}$ denotes the largest square divisor of $N$. 

\begin{lem}\label{lem:red_det_constraints} Let \(\ell\) be a prime of good reduction for \(A\) and suppose that $\oAl$ sits in the sequence \eqref{red_extn}. Let \(p\neq \ell\) be a good prime for \(A\) and let $f$ denote the order of $p$ in $\left(\zz/N_{\sq}\zz\right)^\times$.
Then there exists $0 \leq x \leq \dim X_{\ell}$ (respectively, $0 \leq y \leq \dim Y_\ell$) such that the value of $\det X_{\ell}$ (respectively, $\det Y_{\ell}$) evaluated at $\Frob_p^{\gcd(f, 120)}$ is $p^{\gcd(f,120)x}$ (respectively, $p^{\gcd(f, 120)y}$).
\end{lem}

\begin{proof}
Since \(\ell\) is a good prime and \(X_\ell\) is composed of Jordan--H\"older factors of \(\oAl\), Corollary~\ref{C:subqtdet} constrains its determinant.
We have \(\det X_\ell = \epsilon \cyc_\ell^x\) for some character \(\epsilon \colon G_\qq \to \bar{\ff}_\ell\) unramified at \(\ell\), and \(0 \leq x \leq \dim X_\ell\), and \(\epsilon^{120}=1\). 
Hence the value of $\det X_{\ell}$ evaluated at $\Frob_p^{120}$ is $\cyc_{\ell}(\Frob_p)^{120x} = p^{120x}$.

In fact, we can do slightly better.  
Since $\det \oAl \simeq \cyc_{\ell}^2$, we have $\det Y_{\ell} \simeq \epsilon^{-1}\cyc_{\ell}^{2-x}$. Since the conductor is multiplicative in extensions, we conclude that $\cond(\epsilon)^2 \mid N$. 
By class field theory, the character $\epsilon$ factors through $\left(\zz/\cond(\epsilon)\zz\right)^\times$, and hence through $\left(\zz/N_{\sq}\zz\right)^\times$, sending $\Frob_p$ to $p \bmod{N_{\sq}}$.  Since $p^f \equiv 1 \bmod{N_{\sq}}$, we have that 
$\epsilon(\Frob_p)^{\gcd(f, 120)}=1$, 
and the value of $\det X_{\ell}$ evaluated at $\Frob_p^{\gcd(f, 120)}$ is $p^{\gcd(f,120)x}$. Exchanging \(X_\ell\) and \(Y_\ell\), we deduce the  result for \(Y_\ell\).
\end{proof}

This is often enough information to find all \(\ell\) for which \(\oAl\) has a nontrivial subquotient.  Namely, by Theorem \ref{integrality}, every root of \(P_p(t)\) has complex absolute value \(p^{1/2}\).  Thus the \(\gcd(f, 120)\)-th power of each root has complex absolute value \(p^{\gcd(f, 120)/2}\), and hence is never \textit{integrally} equal to \(1\) or \(p^{\gcd(f, 120)}\). 
Since Lemma \ref{lem:red_det_constraints} guarantees that this equality must hold modulo \(\ell\) for any good prime \(\ell\) for which \(\oAl\) is reducible with a \(1\)-dimensional subquotient, we always get a nontrivial condition on \(\ell\).  Some care must be taken to rule out \(\ell\) for which \(\oAl\) only has \(2\)-dimensional subquotient(s).

\subsubsection{{Odd-dimensional subquotient (cf.~\cite[Section 3.1]{Dieulefait})}}

Let $p$ be a good prime.   Given a polynomial \(P(t)\) and an integer \(f\), write $P^{(f)}(t)$ for the polynomial whose roots are the $f$th powers of roots of $P(t)$.  Universal formulas for such polynomials in terms of the coefficients of \(P(t)\) are easy to compute, and are implemented in our code in the case where \(P\) is a degree \(4\) polynomial whose roots multiply in pairs to \(p^\alpha\), and \(f \mid 120\).  

\begin{algo}[{cf.~\cite[Section 3.1]{Dieulefait}}]\label{alg:odd}
Given a typical genus \(2\) Jacobian \(A/\qq\) of conductor $N$, let $f$ denote the order of $p$ in $\left(\zz/N_{\sq}\zz\right)^\times$ and write $f' = \gcd(f, 120)$. Compute an integer $M_{\text{odd}}$ as follows.
\begin{enumerate}[\upshape (1)]
    \item Choose a nonempty finite set \(\mathcal{T}\) of auxiliary good primes \(p \nmid N\).
    \item For each \(p\), compute
    \[ R_p \colonequals P_p^{(f')}(1).\]    
    \item Let \(M_{\text{odd}} = \gcd_{p \in \mathcal{T}}(p R_p)\) over all auxiliary primes.
\end{enumerate}
Return the list of prime divisors \(\ell\) of \(M_{\text{odd}}\).
\end{algo}

\begin{rem}Algorithm~\ref{alg:odd} offers a modest improvement on \cite[Section 3.1]{Dieulefait}), where the exponent \(f\) is used (without taking the gcd with \(120\).) \end{rem}

\begin{prop}\label{prop:alg_odd_output}
Any good prime \(\ell\) for which $\oAl$ has an odd-dimensional subrepresentation is returned by Algorithm~\ref{alg:odd}.
\begin{proof}
Since $\oAl$ is $4$-dimensional and has an odd-dimensional subrepresentation, it has a $1$-dimensional subquotient. For any $p \in \mathcal{T}$, Lemma~\ref{lem:red_det_constraints} shows that the value of $\det X_{\ell}$ evaluated at $\operatorname{Frob}_p^{f'}$ is either $p^{f'}$ or $1$. Thus, the action of $\operatorname{Frob}_p^{f'}$ on $\oAl$
has an eigenvalue that is equal to $p^{f'}$ or $1$ in $\bar{\ff}_\ell$, and so \(P_p^{(f')}(t)\) has a root that is equal to \(1\) or \(p^{f'}\) in \(\bar{\ff}_\ell\). Since the roots of \(P^{(f')}(t)\) multiply in pairs to \(p^{f'}\), we have \(P_p^{(f')}(p^{f'}) = p^{2f'}P_p^{(f')}(1)\). 
Hence \(\ell\) divides \(p \cdot P_p^{(f')}(1)=pR_p\).
\end{proof}
\end{prop}

Using Theorem \ref{integrality}, we can give a theoretical bound on the ``worst case" of this step of the algorithm using only one auxiliary prime \(p\).  Of course, taking the greatest common divisor over multiple auxiliary primes will likely remove extraneous factors, and in practice this step of the algorithm runs substantially faster than other steps.

\begin{prop}\label{prop:alg_odd_terminates}
Algorithm \ref{alg:odd} terminates.  More precisely, if \(p\) is good, then
\[0 \neq |M_{\text{odd}}| \ll p^{241},\]
where the implied constant is absolute.
\end{prop}
\begin{proof}
This follows from the fact that the coefficient of \(t^i\) in \(P_p^{(f')}(t)\) has magnitude on the order of \(p^{(2-i/2)f'}\) and \(f' \leq 120\). 
\end{proof}

\subsubsection{Two-dimensional subquotients}
We now assume that $\oAl$ is reducible, but does not have any odd-dimensional subquotients.  In particular, it has an irreducible subrepresentation $X_{\ell}$ of dimension $2$, with irreducible quotient \(Y_\ell\) of dimension \(2\).
If $\oAl$ is reducible but indecomposable,
 then $X_{\ell}$ is the unique subrepresentation of $\oAl$ and $Y_\ell^\vee \otimes \cyc_{\ell}$ is the unique subrepresentation of $\overline{A[\ell]}^\vee \otimes \cyc_\ell$. The isomorphism \(T_\ell A \simeq (T_\ell A)^\vee \otimes \cyc_\ell\) from \eqref{E:WeilpairingSim}
 yields an isomorphism \(\oAl \simeq (\oAl)^\vee \otimes \cyc_\ell\) and hence
\( X_\ell \simeq Y_\ell^\vee \otimes \cyc_\ell\).
Otherwise, $\oAl \simeq X_\ell \oplus Y_\ell$ and so the nondegeneracy of the Weil pairing 
gives
\[X_\ell \oplus Y_\ell \simeq \left(X_\ell^\vee \otimes \cyc_\ell\right) \oplus \left(Y_\ell^\vee \otimes \cyc_\ell\right).\]
Therefore either:
\begin{enumerate}[(a)]
\item\label{case:related} $X_\ell \simeq Y_\ell^\vee \otimes \cyc_\ell$ and $Y_\ell\simeq X_\ell^\vee \otimes \cyc_\ell$, or 
\item\label{case:selfdual} $X_\ell \simeq X_\ell^\vee \otimes \cyc_\ell$ and $Y_\ell \simeq Y_\ell^\vee \otimes \cyc_\ell$ and $\oAl\simeq X_\ell \oplus Y_\ell$.
\end{enumerate}
We call the first case \defi{related $2$-dimensional subquotients} and the second case \defi{self-dual $2$-dimensional subrepresentations}.  We will see that the ideas of Lemma \ref{lem:red_det_constraints} easily extend to treat the related subquotient case; we will use the validity of Serre's conjecture to treat the self-dual case. In the case that \(\oAl\) is decomposable, the above two cases correspond respectively to the index \(2\) subgroup \(M_\ell\) in cases \eqref{max:congruence} (the isotropic case) and \eqref{max:quadric} (the nondegenerate case) of Lemma~\ref{lem:max_subgroups}. 

\subsubsection{Related two-dimensional subquotients  (cf.~\cite[Section 3.2]{Dieulefait})}

Let $p$ be a good prime. Let $P_p(t) \colonequals t^4 - at^3 + bt^2 - p a t + p^2$ be the characteristic polynomial of $\Frob_p$ acting on $\oAl$. Suppose that $\alpha$ and $\beta$ are the eigenvalues of $\Frob_p$ acting on the subrepresentation $X_\ell$.  Then, since $X_\ell \simeq Y_\ell^\vee \otimes \cyc_\ell$, the eigenvalues of the action of $\Frob_p$ on $Y_{\ell}$ are $p/\alpha$ and $p/\beta$.  The action of $\Frob_p$ on $\det X_{\ell}$ is therefore by a product of two of the roots of $P_p(t)$ that do not multiply to $p$. Note that there are four such pairs of roots of $P_p(t)$ that do not multiply to $p$. Let $Q_p(t)$ be the quartic polynomial whose roots are the products of pairs of roots of $P_p(t)$ that do not multiply to $p$.  By design, the roots of \(Q_p(t)\) have complex absolute value \(p\), but are not equal to \(p\).  (It is elementary to work out that
\[Q_p(t) = t^4 - (b - 2p)t^3 + p(a^2 - 2b + 2p)t^2 - p^2(b - 2p)t + p^4\]
and is a quartic whose roots multiply in pairs to $p^{2}$.)

\begin{algo}[{cf.~\cite[Section 3.2]{Dieulefait}}]\label{alg:related}
Given a typical genus \(2\) Jacobian \(A/\qq\) of conductor $N$, let $f$ denote the order of $p$ in $\left(\zz/N_{\sq}\zz\right)^\times$ and write $f' = \gcd(f, 120)$. Compute an integer $M_{\text{related}}$ as follows.
\begin{enumerate}[\upshape (1)]
    \item Choose a finite set \(\mathcal{T}\) of auxiliary good primes \(p \nmid N\);
    \item For each \(p\), compute the product
    \[ R_p \colonequals Q_p^{(f')}(1) Q_p^{(f')}(p^{f'}) \]
    \item Let \(M_{\text{related}} = \gcd_{p \in \mathcal{T}}(p R_p)\). 
\end{enumerate}
Return the list of prime divisors \(\ell\) of \(M_{\text{related}}\).
\end{algo}

\begin{rem} Algorithm~\ref{alg:related} offers a modest improvement on the procedure described in \cite[Section 3.2]{Dieulefait}) by taking the gcd of \(f\) with \(120\). \end{rem}

\begin{prop}\label{prop:alg_related_output}
Any good prime \(\ell\) for which $\oAl$ has related two-dimensional subquotients is returned by Algorithm~\ref{alg:related}.
\begin{proof}
Proceed similarly as in the proof of Proposition~\ref{prop:alg_odd_output} --- in particular, $\ell$ divides $Q_p^{(f')}(1)$, $Q_p^{(f')}(p^{f'})$, or $Q_p^{(f')}(p^{2f'})$ and hence $\ell$ divides $p R_p$ since \(Q_p^{(f')}(p^{2f'}) = p^{4f'} Q_p^{(f')}(1)\).
\end{proof}
\end{prop}

\noindent A theoretical ``worst case" analysis yields the following.

\begin{prop}\label{prop:alg_related_terminates}
Algorithm \ref{alg:related} terminates.  More precisely, if \(q\) is the smallest surjective prime for \(A\), then a good prime \(p\) for which \(R_p\) is nonzero is bounded by a function of \(q\).
Assuming GRH, 
\[p \ll q^{22}\log^2(q N),\]
where the implied constants are absolute and effectively computable.
Moreover, for such a prime \(p\), 
\[|M_{\text{related}}| \ll p^{961} \ll q^{21142}\log^{1922}(q N),\]
where the implied constants are absolute.
\end{prop}
\begin{proof}
By Serre's open image theorem for genus \(2\) curves, such a prime \(q\) exists. Since there exists an element of \(\GSp_{4}(\ff_q)\) with irreducible characteristic polynomial, by Lemma \ref{lem:effective_cheb} there exists a prime \(p\) (bounded as claimed) such that \(R_p\) is nonzero modulo \(q\). Finally, 
\[M_\text{related} \leq pR_p = p Q^{(f')}(1) Q^{(f')}(p^{f'}) \ll p^{8f'+1} \ll p^{961},\]
since the coefficient of \(t^i\) in \(Q^{(f')}(t)\) has magnitude on the order of \(p^{(4-i)f'}\) and \(f' \leq 120\).
\end{proof}

\subsubsection{Self-dual two-dimensional subrepresentations (cf.~\cite[Section 3.3]{Dieulefait})}

In this case, both subrepresentations $X_\ell$ and $Y_{\ell}$ are absolutely irreducible $2$-dimensional Galois representations with determinant the cyclotomic character $\cyc_\ell$.  It follows that the representations are odd (i.e., the determinant of complex conjugation is $-1$.) Therefore, by the Khare--Wintenberger theorem (formerly Serre's conjecture on the modularity of mod-$\ell$ Galois representations) \cite{khare2006, kharewintenberger2009I, kharewintenberger2009II}, both $X_\ell$ and $Y_\ell$ are \defi{modular}; that is, for $i = 1,2$, there exist newforms $f_i \in S_{k_i}^{\text{new}}(\Gamma_1(N_i), \epsilon_i)$ such that
\begin{equation*}
X_\ell \cong \bar{\rho}_{f_1, \ell} \mbox{ and } Y_\ell \cong \bar{\rho}_{f_2, \ell}.
\end{equation*}
Furthermore, by the multiplicativity of Artin conductors, we obtain the divisibility $N_1N_2 \mid N$.

\begin{lem}
Both $f_1$ and $f_2$ have weight two and trivial Nebentypus; that is, $k_1 = k_2 = 2$, and $\epsilon_1 = \epsilon_2 = 1$.
\end{lem}

\begin{proof}
From Theorem~\ref{T:serre-raynaud}, we have that $X_\ell|_{I_\ell}$ and $Y_\ell|_{I_\ell}$ must each be conjugate to either of the following subgroups of $\GL_2(\ff_\ell)$:

    \[ \begin{pmatrix}
1 & *\\
0 & \cyc_\ell
\end{pmatrix} \mbox{   or   } 
\begin{pmatrix}
\psi_2 & 0\\
0 & \psi_2^\ell
\end{pmatrix}.\]

The assertion of weight $2$ now follows from \cite[Proposition 3]{serreconj}. (Alternatively, one may use Proposition $4$ of \emph{loc. cit.}, observing that $X_\ell$ and $Y_\ell$ are finite and flat as group schemes over $\zz_\ell$ because $\ell$ is a prime of good reduction.)

From Section~$1$ of \emph{loc. cit.}, the Nebentypus $\epsilon_i$ of $f_i$ satisfies, for all $p \nmid \ell N$,
\[ \det X_\ell(\Frob_p) = p \cdot \epsilon_i(p), \]
where this equality is viewed inside $\bar{\ff}_\ell^\times$. The triviality of \(\epsilon_i\) follows.
\end{proof}

We therefore have newforms $f_i \in \mathcal{S}_2^{\text{new}}(\Gamma_0(N_i))$ such that
\begin{equation}\label{eqn:decompose}
\oAl\simeq \bar{\rho}_{f_1, \ell} \oplus \bar{\rho}_{f_2, \ell}.
\end{equation}
We may assume without loss of generality that \(N_1 \leq \sqrt{N}\).  Let $p \nmid N$ be an auxiliary prime. We obtain from equation~\eqref{eqn:decompose} that the integral characteristic polynomial of Frobenius factors:
\[ P_p(t) \equiv (t^2 - a_p(f_1)t + p)(t^2 - a_p(f_2)t + p) \mod{\ell};\]
here we use the standard property that, for $f$ a normalised eigenform with trivial Nebentypus, $\rho_{f,\ell}(\Frob_p)$ satisfies the polynomial equation $t^2 - a_p(f)t + p$ for $p \neq \ell$. In particular, we have
\[ \Res(P_p(t),t^2 - a_p(f_1)t + p) \equiv 0 \mod{\ell}. \]
This serves as the basis of the algorithm to find all primes $\ell$ in this case.

\begin{algo}[{\cite[Section 3.3]{Dieulefait}}]\label{alg:cuspforms}
Given a typical genus \(2\) Jacobian \(A/\qq\) of conductor $N$, compute an integer $M_{\text{self-dual}}$ as follows.

\begin{enumerate}[\upshape (1)]
    \item Compute the set $S$ of divisors $d$ of $N$ with $d \leq \sqrt{N}$.
    \item For each $d \in S$:
    \begin{enumerate}[\upshape (a)]
        \item choose a finite set \(\mathcal{T}\) of auxiliary primes $p \nmid N$;
        \item for each auxiliary prime $p$, compute the \emph{Hecke $L$-polynomial}
       \[ Q_d(t) \colonequals \prod_{f}(t^2 - a_p(f)t + p), \]
    where the product is taken over the finitely many newforms in $\mathcal{S}_2^{\text{new}}(\Gamma_0(d))$;
        \item compute the resultant
        \[ R_p(d) \colonequals \Res(P_p(t),Q_d(t));\]
        \item Take the greatest common divisor
        \[ M(d) \colonequals \gcd_{p \in \mathcal{T}}(pR_p(d)).  \]
    \end{enumerate}
    \item Let $M_{\text{self-dual}} \colonequals \prod_{d \in S}M(d)$.
\end{enumerate}
Return the list of prime divisors $\ell$ of $M_{\text{self-dual}}$.
\end{algo}

\begin{prop}\label{prop:alg_cuspforms_output}
Any good prime \(\ell\) for which $\oAl$ has self-dual \(2\)-dimensional subrepresentations is returned by Algorithm~\ref{alg:cuspforms}.
\end{prop}
\begin{proof}
As explained before Algorithm~\ref{alg:cuspforms},
there exists $N_1 \in S$ and a newform $f_1 \in \mathcal{S}_2^{\text{new}} (\Gamma_0(N_1))$ such that $\operatorname{Res}(P_p(t), t^2-a_p f_1 t + p) \equiv 0 \bmod{\ell}$ for every $p \in \mathcal{T} \smallsetminus \{\ell\}$. In particular, $pR_p(N_1) \equiv 0 \bmod{\ell}$, so $\ell$ divides $M(N_1)$ and $M_{\text{self-dual}}$.
\end{proof}

We can again do a ``worst case" theoretical analysis of this algorithm to conclude the following.  As this indicates, this is by far the limiting step of the algorithm.

\begin{prop}\label{prop:alg_cuspforms_terminates}
Algorithm \ref{alg:cuspforms} terminates.  More precisely, if \(q\) is the smallest surjective prime for \(A\), then a good prime \(p\) for which \(R_p(d)\) is nonzero is bounded by a function of \(q\).
Assuming GRH, 
\(p \ll q^{22}\log^2(q N)\),
where the implied constant is absolute and effectively computable.
Moreover, for such a prime \(p\), we have 
\[|R_p(d)| \ll (2p^{1/2})^{8 \dim \mathcal{S}^{\text{new}}_2(\Gamma_0(d))} \ll (4p)^{(d+1)/3},\]
and so all together 
\[|M_{\text{self-dual}}| \ll (4p)^{N^{1/2 + \epsilon}},\]
where the implied constants are absolute. 
\end{prop}
\begin{proof}
As in Proposition \ref{prop:alg_related_terminates}, we use Serre's open image theorem and the Effective Chebotarev Theorem.  If \(R_p(d)\) is zero integrally, then in particular \(R_p(d) \equiv 0 \bmod{q}\) and \(P_p(t) \bmod q\) has a factor in common with a quadratic polynomial is therefore reducible modulo \(q\).  Since \(\GSp_4(\ff_q)\) contains elements that do not have reducible characteristic polynomial, Lemma \ref{lem:effective_cheb} implies that such elements are the image of \(\Frob_p\) for \(p\) bounded as claimed.

The resultant \(R_p(d)\) is the product of the pairwise differences of the roots of \(P_p(t)\) and \(Q_d(t)\), which all have complex absolute value \(p^{1/2}\).  Hence the pairwise differences have absolute value at most \(2 p^{1/2}\).  Moreover \(\dim\mathcal{S}^\text{new}_2(\Gamma_0(d)) \leq (d+1)/12\) by \cite[Theorem 2]{martin2005dimensions}. Since there are \(8\dim \mathcal{S}^\text{new}_2(\Gamma_0(d))\) such terms multiplied to give \(R_p(d)\), the bound for \(R_p(d)\) follows.  Since \(M_{\text{self-dual}} = \prod_{\substack{d \mid N \\ d \leq \sqrt{N}}} p R_p(d),\) it suffices to bound
\[\sum_{\substack{d \mid N \\ d \leq \sqrt{N}}} \frac{d + 4}{3} \leq \sum_{\substack{d \mid N \\ d \leq \sqrt{N}}} \frac{\sqrt{N} + 4}{3} \leq \sigma_0(N) \frac{\sqrt{N} + 4}{3}. \]
Since \(\sigma_0(N) \ll N^{\epsilon}\) by \cite[(31) on page 296]{Apostol}, we obtain the claimed bound.
\end{proof}

\begin{rem}
The polynomial $Q_d(t)$ in step (2) of Algorithm~\ref{alg:cuspforms} is closely related to the characteristic polynomial $H_d(z)$ of the Hecke operator $T_p$ acting on the space $S_2(\Gamma_0(d))$, which may be computed via modular symbols computations. One may recover $Q_d(t)$ from $H_d(z)$ by first homogenizing $H$ with an auxiliary variable $t$ (say) to obtain $H_d(z,t)$, and setting $z = p + t^2$ (an observation we made in conjunction with Joseph Wetherell).
\end{rem}

\begin{rem}
In our computations of nonsurjective primes for the database of typical genus~$2$ curves with conductor at most \(2^{20}\) (including those in the LMFDB), we only needed to use polynomials $Q_d(t)$ for level up to $2^{10}$ (since step (1) of the Algorithm has a $\sqrt{N}$ term). We are grateful to Andrew Sutherland for providing us with a precomputed dataset of the characteristic polynomials of the Hecke operators for these levels, resulting from the creation of an extensive database of modular forms going well beyond what was previously available \cite{best2020computing}.
\end{rem}

\begin{rem}
Our Sage implementation uses two auxiliary primes in Step 2(b) of the above algorithm. Increasing the number of such primes yields smaller supersets at the expense of longer runtime.
\end{rem}

\subsection{Good primes that are geometrically irreducible} \label{S:good_primes_geometrically_irreducible}

Let \(\phi\) be any quadratic Dirichlet character \(\phi \colon (\zz/N\zz)^\times \to \{\pm 1\}\).
Our goal in this subsection is to find all good primes \(\ell\) \defi{governed by \(\phi\)}, by which we mean that
\[\tr (\rho_{A, \ell}(\Frob_p)) \equiv a_p \equiv 0 \bmod \ell \] 
whenever \(\phi(p) = -1\). 
 (Recall that \(-a_p\) is the coefficient of \(t^3\) in \(P_p(t)\)).

We will consider the set of all quadratic Dirichlet character \(\phi \colon (\zz/N\zz)^\times \to \{\pm 1\}\).  Using the structure theorem for finite abelian groups and the fact that \(\phi\) factors through \((\zz/N\zz)^\times / (\zz/N\zz)^{\times 2}\), this set has the structure of an \(\ff_2\)-vector space of dimension 
\[d(N) \colonequals \omega(N) + \begin{cases} 0 & : \ v_2(N)=0 \\
 -1 & : \ v_2(N)=1 \\ 0 & : \ v_2(N) = 2 \\ 1 & : \ v_2(N) \geq 3, \end{cases}\]
where \(\omega(m)\) denotes the number of prime factors of \(m\) and \(v_2(m)\) is the \(2\)-adic valuation of \(m\). In particular, \(d(N) \leq \omega(N) + 1\).

\begin{algo}[{\cite[Sections 3.4-3.5]{Dieulefait}}]\label{alg:quad_char}
Given a typical genus \(2\) Jacobian \(A/\qq\) of conductor $N$, compute an integer $M_{\text{quad}}$ as follows.

\begin{enumerate}[\upshape (1)]
    \item Compute the set $S$ of quadratic Dirichlet characters \(\phi \colon (\zz/N\zz)^\times \to \{\pm 1\}\).
    \item For each $\phi \in S$:
    \begin{enumerate}[\upshape (a)]
        \item Choose a nonempty finite set \(\mathcal{T}\) of ``auxiliary'' primes $p \nmid N$ for which \(a_p \neq 0\) and \(\phi(p) = -1\).
        \item Take the greatest common divisor
        \[ M_\phi \colonequals \gcd_{p \in \mathcal{T}}(pa_p),  \]
        over all auxiliary primes \(p\).
    \end{enumerate}
    \item Let $M_{\text{quad}} \colonequals \prod_{\phi \in S}M_\phi$.       
\end{enumerate}
    Return the list of prime divisors $\ell$ of $M_{\text{quad}}$.
\end{algo}

\begin{prop}\label{prop:alg_quad_char_output}
Any good prime \(\ell\) for which $\ell$ is governed by a quadratic character is returned by Algorithm~\ref{alg:quad_char}.
\end{prop}
\begin{proof}
Suppose that $\ell$ is governed by the quadratic character $\phi\colon (\mathbb{Z}/N\mathbb{Z})^\times \to \{\pm 1\}$.
Then for every good prime \(p \neq \ell\) for which \(\phi(p)=-1\), the prime \(\ell\) must divide the integral trace \(a_p\) of Frobenius.  Hence \(\ell\) divides \(M_\phi\) and \(M_{\text{quad}}\).
\end{proof}

\begin{prop}\label{prop:alg_quad_char_terminates}
Algorithm \ref{alg:quad_char} terminates.  More precisely, if \(q\) is the smallest surjective prime for \(A\), then a good prime \(p\) for which \(\phi(p)=-1\) and \(a_p\) is nonzero is bounded by a function of \(q\).
Assuming GRH, 
\(p \ll 2^{2d(N)}q^{22}\log^2(qN)\),
where the implied constant is absolute and effectively computable.
Moreover, we have 
\[\prod_{ \phi \in S }\prod_{ \substack{\text{\(\ell\) governed} \\ \text{by \(\phi\)}} } \ell \ll (2^{3d(N)}q^{33}\log ^3(qN))^{2 - 2^{1-d(N)}} \ll 2^{6\omega(N)}q^{66}\log ^6(qN),\]
where the implied constant is absolute and effectively computable.
\end{prop}
\begin{proof}
We imitate the proof of \cite[Lemma 21]{LarsonVaintrob} in our setting.  Let \(V\) be the \(d\)-dimensional \(\ff_2\)-vector space of quadratic Dirichlet characters of modulus \(N\) (equivalently, quadratic Galois characters unramified outside of \(N\)).  Let \(\rho_V \colon G_K \to V^\vee\) denote the representation sending \(\Frob_p\) to the linear functional \(\phi \mapsto \phi(p)\).  Since the character for \(\PGSp_4(\ff_q)/\PSp_4(\ff_q)\) is the abelianization of \(\pp \rho_{A, q}\), we conclude in the same way as \cite[Proof of Lemma 21]{LarsonVaintrob} that for any \(\alpha \in V^\vee\), there exists an \(X_\alpha \in \GSp_4(\ff_q)\) with \(\tr(X_\alpha) \neq 0\) such that \((\alpha, X_\alpha)\) is in the image of \(\rho_V \times \rho_{A, \ell}\).

Apply the effective Chebotarev density theorem to the Galois extension corresponding to \(\rho_V \times \rho_{A,q}\).  This has degree at most \(2^{d(N)}|\GSp_4(\ff_q)|\) and is unramified outside of \(qN\).  Therefore, assuming GRH and combining~\eqref{ECDT-GRH} and \eqref{logdK-bound}, there exists a prime 
\[p_{\alpha} \ll 2^{2d(N)}q^{22}\log^2(qN)\] 
for which \((\alpha, X_\alpha) = (\rho_V(\Frob_{p_\alpha}), \rho_{A, q}(\Frob_{p_\alpha})) \). 
Let \(\phi\) be a character not in the kernel of \(\alpha\).  Any exceptional prime \(\ell\) governed by \(\phi\) must divide \(p_{\alpha} a_{p_\alpha}\), which is nonzero because it is nonzero modulo \(q\).  This proves that the algorithm terminates, since every \(\phi\) is not in the kernel of precisely half of all \(\alpha \in V^\vee\).  We now bound the size of the product of all \(\ell\) governed by a character in \(S\).  If \(\ell\) is governed by \(\phi\), then \(\ell\) divides the quantity
\[p|a_p|\leq p^{3/2} \ll  2^{3d(N)}q^{33}\log^3(qN).\]
Taking the product over all nonzero \(\alpha\) in \(V\) (of which there are \(2^{d(N)}-1\)), each \(\ell\) will show up half the time, so we obtain:
\[\left(\prod_{ \substack{\text{\(\ell\) governed} \\ \text{by \(\phi \in S\)}} } \ell \right)^{2^{d(N)-1}} \ll \left(2^{3d(N)}q^{33}\log^3(qN) \right)^{2^{d(N)}-1}, \]
which implies the result by taking the \((2^{d(N)-1})\)th root of both sides.
\end{proof}

Putting all of these pieces together, we obtain the following.

\begin{proof}[Proof of Theorem~\ref{T:mainthm}\eqref{part1}]
If \(\rho_{A, \ell}\) is nonsurjective, \(\ell > 7\), and \(\ell \nmid N\), then Proposition~\ref{prop:max_semistable} implies that \(\rho_{A, \ell}(G_\qq)\) must be in one of the maximal subgroups of Type~\eqref{max:red} or \eqref{max:irred} listed in Lemma \ref{lem:max_subgroups}. If it is contained in one of the reducible subgroups, i.e. the subgroups of Type~\eqref{max:red}, then \(\rho_{A, \ell}(G_\qq)\) (and, hence, \(\rho_{A, \ell}(G_{\qq}) \otimes \bar{\ff}_\ell\)) is reducible, and so \(\ell\) is added to \(\PossiblyNonsurjectivePrimes\) in Step~\eqref{step:non_abs_irred} by Propositions~\ref{prop:alg_odd_output}, \ref{prop:alg_related_output}, and \ref{prop:alg_cuspforms_output}. If \(\rho_{A,\ell}(G_\qq)\) is contained in one of the index \(2\) subgroups \(M_\ell\) of an irreducible subgroup of Type~\eqref{max:irred} listed in Lemma \ref{lem:max_subgroups}, then again \(\ell\) is added to \(\PossiblyNonsurjectivePrimes\) in Step~\eqref{step:non_abs_irred}, since \(M_\ell \otimes \bar{\ff}_\ell\) is always reducible by Lemma~\ref{L:conseqclass}\eqref{L:parsubgp}.

Hence we may assume that \(\rho_{A, \ell}(G_\qq)\) is contained in one of the irreducible maximal subgroups \(G_\ell\) of Type~\eqref{max:irred} listed in Lemma \ref{lem:max_subgroups}, but not in the index \(2\) subgroup \(M_\ell\).  The normalizer character 
\[G_\qq \xrightarrow{\rho_{A, \ell}} G_\ell \to G_\ell / M_\ell  = \{\pm 1\}\]
is nontrivial and unramified outside of \(N\), and so it corresponds to a quadratic Dirichlet character \(\phi \colon (\zz/N \zz)^\times \to \{\pm 1\}\). Lemma~\ref{L:conseqclass}\eqref{L:tracezero_GSp4poly} shows that \(\tr(g) =  0\) in $\mathbb{F}_\ell$ for any \(g \in G_\ell \smallsetminus M_\ell\). Consequently, $\ell$ is governed by $\phi$ (in the language of Section~\ref{S:good_primes_geometrically_irreducible}), so \(\ell\) is added to \(\PossiblyNonsurjectivePrimes\) in Step~\eqref{step:quad_char} by Proposition~\ref{prop:alg_quad_char_output}.
\end{proof}

\subsection{Bounds on Serre's open image theorem}\label{ssec:mainbound}

In this section we combine the theoretical worst case bounds in the Algorithms \ref{alg:odd}, \ref{alg:related}, \ref{alg:cuspforms}, and \ref{alg:quad_char} to give a bound on the smallest surjective good prime \(q\), and the product of all nonsurjective primes, thereby establishing Theorem~\ref{T:mainbound}.

\begin{cor}
Let \(A/\qq\) be a typical genus \(2\) Jacobian of conductor \(N\).
Assuming GRH, we have
\[\prod_{\ell \text{ nonsurjective}} \ell \ll \exp(N^{1/2 + \epsilon}),\]
where the implied constant is absolute and effectively computable.
\end{cor}
\begin{proof}
Let \(q\) be the smallest surjective good prime for \(A\), which is finite by Serre's open image theorem.  Multiplying the bounds in Propositions \ref{prop:alg_odd_terminates}, \ref{prop:alg_related_terminates}, \ref{prop:alg_cuspforms_terminates}, and \ref{prop:alg_quad_char_terminates} by the conductor \(N\), the product of all nonsurjective primes is bounded by a function of \(q\) and \(N\) of the following shape
\begin{equation}\label{eq:bound_q}
    \prod_{\ell \text{ nonsurjective}} \ell \ll q^{N^{1/2 + \epsilon}}.
\end{equation}
On the other hand, since \(q\) is the smallest surjective prime by definition, the product of all primes less than \(q\) divides the product of all nonsurjective primes.  Using \cite[Lemme 11]{serre81}, we have 
\[ \exp(q) \ll \prod_{\ell < q} \ell \leq\prod_{\ell \text{ nonsurjective}} \ell \ll q^{N^{1/2 + \epsilon}}.\]
Combining the first and last terms, we have
\( q \ll N^{1/2 + \epsilon}\log(q)\),
whence \(q \ll N^{1/2 + \epsilon}\).  Plugging this back into \eqref{eq:bound_q} yields a bound of \((N^{1/2 + \epsilon})^{N^{1/2 + \epsilon}}\).  By taking logarithms and using that \(\log N \ll N^{\epsilon}\), one sees that \((N^{1/2 + \epsilon})^{N^{1/2 + \epsilon}} \ll \exp(N^{1/2 + \epsilon})\), and the claimed bound follows.
\end{proof}

\section{Testing surjectivity of \texorpdfstring{$\rho_{A,\ell}$}{rho A ell}}\label{S:test_ell}
In this section we establish Theorem \ref{T:mainthm}\eqref{part2}. The goal is to weed out any extraneous nonsurjective primes in the output $\PossiblyNonsurjectivePrimes$ of Algorithm~\ref{A:combinedPart1algo} to produce a smaller list $\LikelyNonsurjectivePrimes(B)$ containing all nonsurjective primes (depending on a chosen bound $B$) by testing the characteristic polynomials of Frobenius elements up to the bound $B$. If $B$ is sufficiently large (quantified in Section~\ref{S:confidence}), the list $\LikelyNonsurjectivePrimes(B)$ is provably the list of nonsurjective primes.

\begin{algo}\label{A:combinedPart2algo}
Given \(B>0\) and the output $\PossiblyNonsurjectivePrimes$ of Algorithm~\ref{A:combinedPart1algo} for the typical genus $2$ curve with equation $y^2 + h(x) y =f(x)$, output a sublist $\LikelyNonsurjectivePrimes(B)$ of $\PossiblyNonsurjectivePrimes$ as follows. 
\begin{enumerate}[\upshape (1)]
    \item Initialize $\LikelyNonsurjectivePrimes(B)$ as $\PossiblyNonsurjectivePrimes$.
    \item Remove $2$ from $\LikelyNonsurjectivePrimes(B)$ if the size of the Galois group of the splitting field of $4f + h^2$ is $720$.
    \item\label{A:computefrobstep} For each good prime \(p < B\), while $\LikelyNonsurjectivePrimes(B)$ is nonempty: 
    \begin{enumerate}[\upshape (a)]
        \item Compute the integral characteristic polynomial \(P_p(t)\) of \(\Frob_p\).
        \item For each prime \(\ell\) in $\LikelyNonsurjectivePrimes(B)$, run Tests~\ref{T:exc}\eqref{T:1920}, \eqref{T:720}, and \eqref{T:5040} on $P_p(t)$ to rule out $\rho_{A,\ell}(G_\qq)$ being contained in one of the exceptional maximal subgroups.
        \item For each prime \(\ell\) in $\LikelyNonsurjectivePrimes(B)$, run Tests~\ref{T:nonexc}\eqref{T:irrtest} and \eqref{T:tracezerotest} on $P_p(t)$ to rule out $\rho_{A,\ell}(G_\qq)$ being contained in one of the nonexceptional maximal subgroups.
        \item If there exists \(\ell\in \LikelyNonsurjectivePrimes(B)\) for which each of the 5 tests Tests~\ref{T:exc}\eqref{T:1920}--\eqref{T:5040} and Tests~\ref{T:nonexc}\eqref{T:irrtest}--\eqref{T:tracezerotest} have succeeded at least once, remove~\(\ell\).
    \end{enumerate}
    \item Return $\LikelyNonsurjectivePrimes(B)$.
\end{enumerate}
\end{algo}

\begin{rem}
In our implementation of Step~\ref{A:computefrobstep} of this algorithm, we have chosen to only use primes $p$ of good reduction for the curve as auxiliary primes, which is a stronger condition than being a good prime for the Jacobian $A$. More precisely, the primes that are good for the Jacobian but bad for the curve are precisely the prime factors of the discriminant $4f+h^2$ of a minimal equation for the curve that do not divide the conductor $N_A$ of the Jacobian. At such a prime, the reduction of the curve consists of two elliptic curves $E_1$ and $E_2$ intersecting transversally at a single point. Since there are many auxiliary primes $p < B$ to choose from, excluding bad primes for the curve is not a serious restriction, but allows us to access the characteristic polynomial of Frobenius directly by counting points on the reduction of the curve. This is not strictly necessary: one could use the characteristic polynomials of Frobenius for the elliptic curves $E_1$ and $E_2$, which can be computed using the \href{https://doc.sagemath.org/html/en/reference/arithmetic_curves/sage/interfaces/genus2reduction.html}{\texttt{genus2reduction}} module of SageMath. 
\end{rem}

We briefly summarize the contents of this section. 
In Section~\ref{S:grpthycrit}, we first prove a purely group-theoretic criterion for a subgroup of $\GSp_4(\ff_\ell)$ to equal the whole group. Then in Section~\ref{S:surjtesting}, we explain Test~\ref{T:exc} and Test~\ref{T:nonexc}, whose validity follows immediately from  Lemma~\ref{L:conseqclass}\eqref{L:precompjust} and Proposition~\ref{P:grpthy} respectively.
The main idea of these tests is to use auxiliary good primes \(p\neq \ell\) to generate characteristic polynomials in the image of \(\rho_{A, \ell}\). If we find enough types of characteristic polynomials to rule out each proper maximal subgroup of \(\GSp_4(\ff_\ell)\) (cf.~Proposition~\ref{P:grpthy}), then we can conclude that \(\rho_{A, \ell}\) is surjective.  In Section~\ref{S:Just}, we prove Theorems \ref{T:mainthm}\eqref{part2} and  \ref{T:Bbound} that justify this algorithm.

\subsection{A group-theoretic criterion}\label{S:grpthycrit}
We now use the classification of maximal subgroups of $\GSp_4(\ff_\ell)$ described in Section~\ref{S:classification} to deduce a group-theoretic criterion for a subgroup $G$ of $\operatorname{GSp}_4(\mathbb{F}_\ell)$ to be the whole group. This is analogous to \cite[Proposition 19 (i)-(ii)]{Serre72}.
\begin{prop}\label{P:grpthy} Fix a prime $\ell \neq 2$ and a subgroup $G \subseteq \operatorname{GSp}_4(\mathbb{F}_\ell)$ with surjective similitude character. Assume that $G$ is not contained in one of the  exceptional maximal subgroups described in Lemma~\ref{lem:max_subgroups}\eqref{max:exceptional}. Then $G = \operatorname{GSp}_4(\mathbb{F}_\ell)$ if and only if there exists matrices $X, Y \in G$ such that
\begin{enumerate}[\upshape (a)]
    \item\label{prop:irredelem}  the characteristic polynomial of $X$ is irreducible; and
    \item\label{prop:redelem} $\operatorname{trace} Y \neq 0$ and the characteristic polynomial of $Y$ has a linear factor with multiplicity one. 
\end{enumerate}
\end{prop}

\begin{proof}
The `only if' direction follows from Proposition~\ref{P:odds} below, where we show that a {\textit{nonzero}} proportion of elements of $\operatorname{GSp}_4(\mathbb{F}_\ell)$ satisfy the conditions in \eqref{prop:irredelem} and \eqref{prop:redelem}.

Now assume that the group $G$ has elements $X$ and $Y$ as in the statement of the proposition. We have to show that $G = \GSp_4(\ff_\ell)$. By assumption, $G$ is not a subgroup of a maximal subgroup of type \eqref{max:exceptional}. For each of the remaining types of maximal subgroups in Lemma~\ref{lem:max_subgroups}, we will use one of the elements $X$ or $Y$ to rule out $G$ being contained in a subgroup of that type. 

\begin{enumerate}[\upshape (a)]
\item By Lemma~\ref{L:charpolygen}~\eqref{easy_necessary}, every element of a subgroup of type~\eqref{max:red} has a reducible characteristic polynomial. The same is true for elements of type~\eqref{max:twistedcubic} by Lemma~\ref{L:conseqclass}~\eqref{L:twistedcubic}. This is violated by the element $X$, so $G$ cannot be contained in a subgroup of type~\eqref{max:red} or type~\eqref{max:twistedcubic}.

\item Recall the notation used in the description of a type~\eqref{max:irred} maximal subgroups in Lemma~\ref{lem:max_subgroups}. By Lemma~\ref{L:conseqclass}~\ref{L:tracezero_GSp4poly}, every element in $G_\ell \setminus M_{\ell}$ has trace $0$.
By Lemma~\ref{L:charpolygen}~\eqref{L:tracezero_charpolygen}, an element with irreducible characteristic polynomial automatically has nonzero trace.  Hence both \(X\) and \(Y\) have nonzero traces, and so cannot be contained in \(G_\ell \smallsetminus M_\ell\).  We now consider two cases
\begin{enumerate}[\upshape (i) ]
    \item If the two subspaces are individually defined over ${\mathbb{F}_\ell}$, then every element in \(M_\ell\) preserves a two-dimensional subspace and hence has a reducible characteristic polynomial. This is violated by the element $X$.
\item If the two subspaces are permuted by $G_{{\mathbb{F}_\ell}}$, then the action of \(M_\ell\) on the corresponding subspaces $V$ and $V'$ are conjugate. Therefore, every $\mathbb{F}_\ell$-rational eigenvalue for the action of $\Frob_p$ on $V$, also appears as an eigenvalue for the action on $V'$, with the same multiplicity. This is violated by the element $Y$.
\end{enumerate}
Hence $G$ cannot be contained in a maximal subgroup of type~\eqref{max:irred}.

\end{enumerate}

Since any subgroup of $\operatorname{GSp}_4(\mathbb{F}_\ell)$ that is not contained in a proper maximal subgroup of $\operatorname{GSp}_4(\mathbb{F}_\ell)$ must equal $\operatorname{GSp}_4(\mathbb{F}_\ell)$, we are done. \qedhere

\end{proof}

\begin{rem}
\cite[Corollary~2.2]{ReynaKappen} gives a very similar criterion for a subgroup $G$ of $\GSp_4(\ff_\ell)$ to contain $\Sp_4(\ff_\ell)$, namely that it contains a transvection, and also an element with irreducible characteristic polynomial (and hence automatically nonzero trace).
\end{rem}

\subsection{Surjectivity tests}\label{S:surjtesting}

\subsubsection{Surjectivity test for $\ell=2$}\label{smallprimes}\hfill

\begin{prop}\label{P:surjtest2}
Let $A$ be the Jacobian of the hyperelliptic curve $y^2 + h(x) y = f(x)$ defined over $\qq$. Then $\rho_{A,2}$ is surjective if and only if the size of the Galois group of the splitting field of $4f+h^2$ is $720$.
\end{prop}
\begin{proof}
This follows from the fact that $\GSp_4(\ff_2) \cong S_6$ which is a group of size $720$, and that the representation $\rho_{A,2}$ is the permutation action of the Galois group on the six roots of $4f+h^2$.
\end{proof}

\subsubsection{Surjectivity tests for $\ell \neq 2$} \hfill

The tests to rule out the exceptional maximal subgroups rely on the existence of the finite lists $C_{1920}$ and $C_{720}$ (independent of \(\ell\)), and $C_{7, 5040}$ given in Lemma~\ref{L:conseqclass}\eqref{L:precompjust}.

\begin{test}[Tests for ruling out exceptional maximal subgroups of $\GSp_4(\ff_\ell)$ for $\ell \neq 2$]\label{T:exc} \hfill\\
Given a polynomial \(P_p(t) = t^4 - a_pt + b_pt^2 - pa_pt + p^2\) and \(\ell \geq 2\), 
\begin{enumerate}[\upshape (i)]
\item\label{T:1920} $P_p(t)$ passes Test~\ref{T:exc}~\eqref{T:1920}  if \(\ell \equiv \pm 1 \bmod{8}\) or \((a_p^2/p, b_p/p)\) mod \(\ell\) lies outside of  $C_{1920}$ mod \(\ell\).
\item\label{T:720} $P_p(t)$ passes Test~\ref{T:exc}~\eqref{T:720}  if \(\ell \equiv \pm 1 \bmod{12}\) or \((a_p^2/p, b_p/p)\) mod \(\ell\) lies outside of $C_{720}$ mod \(\ell\).
\item\label{T:5040} $P_p(t)$ passes Test~\ref{T:exc}~\eqref{T:5040}  if \(\ell \neq 7\) or \((a_p^2/p, b_p/p)\) mod \(\ell\) lies outside of $C_{7,5040}$.
\end{enumerate}
\end{test}

\begin{test}[Tests for ruling out non-exceptional maximal subgroups for $\ell \neq 2$]\label{T:nonexc}\hfill \\
Given a polynomial \(P_p(t) = t^4 - a_pt + b_pt^2 - pa_pt + p^2\) and \(\ell \geq 2\), 
\begin{enumerate}[\upshape (i)]
    \item\label{T:irrtest}  $P_p(t)$ passes Test~\ref{T:nonexc}~\eqref{T:irrtest}  if $P_p(t)$ modulo $\ell$ is irreducible.
    \item\label{T:tracezerotest}  $P_p(t)$ passes Test~\ref{T:nonexc}~\eqref{T:tracezerotest}  if $P_p(t)$ modulo $\ell$ has a linear factor of multiplicity \(1\) and has nonzero trace.
\end{enumerate}
\end{test}

For any one of the five tests above, say that the test \defi{succeeds} if a given polynomial $P_p(t)$ \defi{passes} the corresponding test.

\begin{rem}\label{R:verbose}
We call
an auxiliary prime \(p\) a \defi{witness} for a given prime \(\ell\) if the polynomial $P_p(t)$ passes one of our tests for $\ell$. The verbose output of our code prints witnesses for each of our tests for each prime \(\ell\) in $\PossiblyNonsurjectivePrimes$ but not in $\LikelyNonsurjectivePrimes(B)$.
\end{rem}

\subsection{Justification for surjectivity tests}\label{S:Just}
Considering Tests \ref{T:exc} and \ref{T:nonexc}, we define
\begin{align*}
C_\alpha &= \{M \in \GSp_4(\FF_\ell) : P_M(t) \text{ is irreducible}\} \\
C_\beta &= \{M \in \GSp_4(\FF_\ell) : \tr(M) \neq 0 \text{ and } P_M(t) \text{ has a linear factor of multiplicity $1$}\} \\
C_{\gamma_1} &= \left \{M \in \GSp_4(\FF_\ell) :  \left(\frac{\tr(M)^2}{\simi(M)},\frac{\midd(M)}{\simi(M)}\right) \not\in  C_{\ell,1920} \text{ or \(\ell \equiv \pm 1 \bmod{8}\)} \right \} \\
C_{\gamma_2} &= \left \{M \in \GSp_4(\FF_\ell) :  \left(\frac{\tr(M)^2}{\simi(M)},\frac{\midd(M)}{\simi(M)}\right) \not\in  C_{\ell,720} \text{ or \(\ell \equiv \pm 1 \bmod{12}\)} \right \} \\
C_{\gamma_3} &= \left \{M \in \GSp_4(\FF_\ell) :  \left(\frac{\tr(M)^2}{\simi(M)},\frac{\midd(M)}{\simi(M)}\right) \not\in  C_{\ell,5040} \text{ or \(\ell \neq 7\)} \right \} \\
C_\gamma &= C_{\gamma_1} \cap C_{\gamma_2} \cap C_{\gamma_3}.          
\end{align*}

\begin{proof}[Proof of Theorem \ref{T:mainthm}\eqref{part2} and Theorem \ref{T:Bbound}]
Let \(B>0\).
Since the list $\LikelyNonsurjectivePrimes(B)$ is a sublist of $\PossiblyNonsurjectivePrimes$, which contains all nonsurjective primes by Theorem~\ref{T:mainthm}\eqref{part1}, any prime which is not in the list $\PossiblyNonsurjectivePrimes$ is surjective. Let
$\ell\in\PossiblyNonsurjectivePrimes$ and \textbf{not} in $\LikelyNonsurjectivePrimes(B)$. If \(\ell=2\), then by Proposition~\ref{P:surjtest2}, $\rho_{A,2}$ is surjective. If \(\ell > 2\), this means that we found primes $p_1,p_2,p_3,p_4, p_5 \leq B$ each distinct from $\ell$ and of good reduction for $A$ for which  $\rho_{A,\ell}(\Frob_{p_1}) \in C_\alpha$, $\rho_{A,\ell}(\Frob_{p_2}) \in C_\beta$, $\rho_{A,\ell}(\Frob_{p_3}) \in C_{\gamma_1}$, $\rho_{A,\ell}(\Frob_{p_4}) \in C_{\gamma_2}$, and $\rho_{A,\ell}(\Frob_{p_4}) \in C_{\gamma_3}$. Note that by \eqref{E:WeilpairingSim}, the similitude factor \(\simi(\rho_{A,\ell}(\Frob_p))\) is \(p\).  Therefore, by Lemma~\ref{L:conseqclass}\eqref{L:precompjust}, it follows that \(\rho_{A, \ell}(G_\qq)\) is not contained in an exceptional maximal subgroup. The surjectivity of $\rho_{A,\ell}$ now follows from  Proposition~\ref{P:grpthy}.

Finally, we will show that if \(B\) is sufficiently large (as quantified by Theorem~\ref{T:Bbound}), then any prime \(\ell\) in $\LikelyNonsurjectivePrimes$ is nonsurjective.
Since the sets $C_\alpha$, $C_\beta$, $C_{\gamma_1}$, $C_{\gamma_2}$ and $C_{\gamma_3}$ are nonempty by Proposition~\ref{P:odds} below and closed under conjugation, it follows from Lemma \ref{lem:effective_cheb} that there exist primes $p_1,p_2,p_3,p_4,p_5 \leq B$ as above.
\end{proof}

\begin{rem}\label{R:GRHAHC}
If we assume both GRH and AHC, Ram Murty and Kumar Murty \cite[p. 52]{MR3025442} noted  (see also \cite[Theorem 2.3]{2001.05428}) that the bound \eqref{ECDT-GRH} can be replaced with  $p \ll \frac{(\log d_K)^2}{|S|}$. Proposition \ref{P:odds}, which follows, shows that the sets $C_\alpha$, $C_\beta$, and $C_\gamma$ have size at least $\frac{|\GSp_4(\FF_\ell)|}{10}$. This can be used to prove the ineffective version of Theorem \ref{T:Bbound} which relies on AHC noted in the introduction in a manner similar to the proof of Theorem \ref{T:Bbound}.
\end{rem}

\section{The probability of success}~\label{S:confidence}
In this section we prove Theorem~\ref{T:Bprob}, by studying the respective probabilities $\alpha_\ell$, $\beta_\ell$, and $\gamma_\ell$ that a matrix chosen uniformly at random from $\GSpFL$ is contained in each of $C_\alpha$, $C_\beta$, and $C_\gamma$ defined in Section~\ref{S:Just}.
\begin{prop}\label{P:odds}
Let $M$ be a matrix chosen uniformly at random from $\GSpFL$ with $\ell$ odd. Then
\begin{enumerate}[ \upshape (i)]
\item The probability that $M \in C_\alpha$ is given by $$\alpha_\ell = \frac{1}{4} - \frac{1}{2(\ell^2+1)}.$$
\item\label{P:odd2} The probability that $M \in C_\beta$ is given by $$\beta_\ell = \frac{3}{8} - \frac{3}{4(\ell-1)} + \frac{1}{2(\ell-1)^2}.$$
\item\label{P:ExpProb} The probability that $M \in C_\gamma$ is $$\gamma_\ell \geq 1 - \frac{3\ell}{\ell^2+1}.$$
\end{enumerate}

\end{prop}


\cite{Shinoda} characterizes all conjugacy classes of elements of $\GSpFL$ for $\ell$ odd, grouping them into 26 different types. For each type $\gamma$, Shinoda further computes the number $N(\gamma)$ of conjugacy classes of type $\gamma$ and the size $|C_{\GSpFL}(\gamma)|$ of the centralizer, which is the size $|C_{\GSpFL}(M)|$ of the centralizer of $M$ in $\GSpFL$ for any $M$ in a conjugacy class of type $\gamma$. The size $|C(\gamma)|$ of any conjugacy class of type $\gamma$ can then easily be computed as 
\(
|\C(\gamma)| =  \frac{|\GSpFL|}{|C_{\GSpFL}(\gamma)|}
\)
and the probability that a uniformly chosen $M \in \GSpFL$ has conjugacy type $\gamma$ is then given by
\begin{equation}
\label{eqn:probA}
\frac{N(\gamma) |\C(\gamma)|}{|\GSpFL|} = \frac{N(\gamma)}{|C_{\GSpFL}(\gamma)|}.
\end{equation}

To prove Proposition \ref{P:odds}, we will need to examine a handful of types of conjugacy classes of $\GSpFL$.
There is only a single conjugacy type $\gamma$ whose characteristic polynomials are irreducible. This type is denoted $K_0$ in \cite{Shinoda} where it is shown there that $N(K_0) = \frac{(\ell-1)(\ell^2-1)}{4}$ and $|C_{\GSpFL}(K_0)| = (\ell-1)(\ell^2 + 1)$.

While there is only one way for a polynomial to be irreducible, there are several ways for a quartic polynomial to have a root of odd order. However, only some of these can occur if $f(t)$ is the characteristic polynomial of a matrix $M \in \GSpFL$ and we only need to concern ourselves with the following three possibilities:

\begin{enumerate}[(a)]
\item\label{case:A} $f(t)$ splits completely over $\FF_\ell$;
\item \label{case:B}$f(t)$ has two roots over $\FF_\ell$, both of which occur with multiplicity one; and
\item $f(t)$ has two simple roots and one double root over $\FF_\ell$. 
\end{enumerate}

Cases~(\ref{case:A}) and (\ref{case:B}) correspond to the conjugacy types $H_0$ and $J_0$ in \cite{Shinoda} respectively. There are two types of conjugacy classes for which $f(t)$ has two simple roots and one double root, which are denoted by $E_0$ and $E_1$ in \cite{Shinoda}.

Data for the relevant conjugacy class types is given by Table \ref{T:conjclasssize}, including the probability that a uniform random $M \in \GSpFL$ has conjugacy type $\gamma$ via \eqref{eqn:probA}.

\begin{table}[h]
\renewcommand{\arraystretch}{1.25}
\begin{tabular}{|c|c|c|c|}
\hline
Type $\gamma$ in \cite{Shinoda} & $N(\gamma)$ &  $|C_{\GSpFL}(\gamma)|$ & Associated Probability\\
\hline
$K_0$ (Irreducible) & $\frac{(\ell-1)(\ell^2-1)}{4}$ &  $(\ell^2 + 1)(\ell-1)$  &  $\frac{1}{4} - \frac{1}{2(\ell^2+1)}$ \\ \hline
$H_0$ (Split) &  $\frac{(\ell-1)(\ell-3)^2}{8}$ & $(\ell -1)^3$ &  $\frac{1}{8} - \frac{1}{2(\ell-1)} + \frac{1}{2(\ell-1)^2}$\\ \hline
$J_0$ (Two Simple Roots) & $\frac{(\ell-1)^3}{4}$ &  $(\ell + 1)(\ell-1)^2$ & $ \frac{1}{4} - \frac{1}{2(\ell+1)}$\\ \hline
$E_0$ (One Double Root)  & $\frac{(\ell-1)(\ell-3)}{2}$ &   $\ell(\ell-1)^2(\ell^2-1)$ &  $\frac{1}{2\ell(\ell^2-1)} - \frac{1}{\ell(\ell-1)(\ell^2-1)}$\\ \hline
$E_1$ (One Double Root)  & $\frac{(\ell-1)(\ell-3)}{2}$ &   $\ell(\ell-1)^2$ & $\frac{1}{2\ell} - \frac{1}{\ell(\ell-1)}$ \\ \hline
\end{tabular}
\vspace{3pt}
\caption{Number of conjugacy classes and centralizer sizes for each conjugacy class type in \cite{Shinoda}.}
\label{T:conjclasssize}
\end{table}

\begin{proof}[Proof of Proposition~\ref{P:odds}]
Part (i) is simply the entry in Table \ref{T:conjclasssize} in the last column corresponding to the ``$K_0$ (Irreducible)" type. 

We now establish part (\ref{P:odd2}). As indicated in the discussion above Table~\ref{T:conjclasssize}, the only conjugacy classes of matrices in $\GSp_4(\ff_\ell)$ whose characteristic polynomials have some linear factors of odd multiplicity are those of the types $H_0, J_0, E_0,E_1$. However, for part~(\ref{P:odd2}) since we are only interested in matrices $M$ also having non-zero trace, it is insufficient to simply sum over the rightmost entries in the bottom four rows of Table \ref{T:conjclasssize}. From \cite[Table 2]{Shinoda}, we see that the elements of $E_0$ and $E_1$ have trace $\frac{c(a+1)^2}{a}$ for some $c,a \in \ff_\ell^\times$ with $a \ne \pm 1$. In particular, it follows that elements of types $E_0$ and $E_1$ have nonzero traces. The elements of type $J_0$ have trace $\frac{(c+a)(c+a^\ell)}{c}$ where $c \in \ff_\ell^\times$ and $a \in \ff_{\ell^2} \setminus \ff_\ell$. Therefore, the elements of type $J_0$ also have nonzero trace. 

It remains to analyze which conjugacy classes of Type $H_0$ have nonzero trace. Following \cite{Shinoda}, the $\frac{(\ell-1)(\ell-3)^2}{8}$ conjugacy classes of type $H_0$ correspond to quadruples of distinct elements in $a_1,a_2,b_1,b_2 \in \FF_\ell^\times$ satisfying $a_1 b_1 = a_2 b_2$ modulo the action of swapping any of $a_1$ with $b_1$, $a_2$ with $b_2$, or $a_1,b_1$ with $a_2, b_2$. The eigenvalues of any matrix in the conjugacy class are $a_1$, $a_2$, $b_1$, and $b_2$.  Consequently the matrix has trace zero only if either $a_2 = -a_1$ and $b_2 = -b_1$ or $b_1 = -a_2$ and $b_2 = -a_1$. This accounts for $\frac{(\ell-1)(\ell-3)}{4}$ of the $\frac{(\ell-1)(\ell-3)^2}{8}$ conjugacy classes of type $H_0$, leaving $\frac{(\ell-1)(\ell-3)(\ell-5)}{8}$ conjugacy classes with non-zero trace. As a result, the probability that a matrix $M \in \GSpFL$ chosen uniformly at random has non-zero trace and totally split characteristic polynomial is
\begin{equation}
\label{eqn:tracezero_split_prob}
\frac{(\ell-1)(\ell-3)(\ell-5)}{8(\ell -1)^3}= \frac{1}{8} - \frac{3}{4(\ell-1)} + \frac{1}{(\ell-1)^2}.
\end{equation}
To obtain part (ii), we add \eqref{eqn:tracezero_split_prob} to the entries in the rightmost column of the final three rows of Table \ref{T:conjclasssize}, getting
\begin{multline*}
\left(    \frac{1}{8} - \frac{3}{4(\ell-1)} + \frac{1}{(\ell-1)^2}\right) + \left ( \frac{1}{4} - \frac{1}{2(\ell+1)} \right ) + \left ( \frac{1}{2\ell(\ell^2-1)} - \frac{1}{\ell(\ell-1)(\ell^2-1)}\right) + \\ \left(\frac{1}{2\ell} - \frac{1}{\ell(\ell-1)}\right)  = \frac{3}{8} - \frac{3}{4(\ell-1)} + \frac{1}{2(\ell-1)^2}.
\end{multline*}
To prove \eqref{P:ExpProb}, note that for any pair $(u,v)$, the cardinality of the set
\[
\{
t^4 - at^3 + bt^2 - am t + m^2 : a,b \in \FF_\ell, m \in \FF_\ell^\times  \text{ and } (\tfrac{a^2}{m},\tfrac{b}{m}) = (u,v) \}
\]
is at most $\ell-1$. By \cite[Theorem 3.5]{MR1440067}, the number of matrices in $\GSp_4(\FF_\ell)$ with a given characteristic polynomial is at most $(\ell+3)^8$. Assuming $\ell \neq 7$, by combining these observations, and noting that $|C_{\ell,720} \cup C_{\ell,1920}| \leq 14$, we obtain the bound
\[ \gamma_\ell \ge 1 - \frac{14 (\ell-1) (\ell+3)^8}{|\GSp_4(\FF_\ell)|}. \]
For $\ell > 17$, this implies the claimed bound. For $3 \leq \ell \leq 17$, we directly check the claim using Magma.
\end{proof}

\begin{lem}
\label{L:CebOnePrime}
Let $C/\qq$ be a typical genus $2$ curve with Jacobian $A$ and suppose $\ell$ is an odd prime such that $\rho_{A,\ell}$ is surjective. For any $\epsilon > 0$, there exists an effective constant $B_0$ (with $B_0 > \ell N_A$) such that for any $B > B_0$ and each $\delta \in \{\alpha, \beta, \gamma\}$, we have
\[
\left|
\frac{|{\{p \text{ prime} : B \leq p \leq 2B \text{ and } \rho_{A,\ell}(\Frob_p) \in C_\delta \}}|}{|\{p \text{ prime} : B \leq p \leq 2B\}| } - \delta_\ell \right| < \epsilon.
\]
\end{lem}
\begin{proof}
Let $G = \Gal(\qq(A[\ell])/\qq)$ and $S \subseteq G$ be any subset that is closed under conjugation. By taking $B$ to be sufficiently large, we have that $B > \ell N_A$ and can make
\[
\left|
\frac{|{\{p \text{ prime} : B \leq p \leq 2B \text{ and } \Frob_p \in S \}}|}{|\{p \text{ prime} : B \leq p \leq 2B\}| } - \frac{|S|}{|G|} \right|
\]
arbitrarily small by \eqref{E:CDT}. Moreover, the previous statement can be made effective by using an effective version of the Chebotarev density theorem; in particular, the value $B_0$ must be larger than the bound $B$ from Equation~\eqref{eq:B_bound}. The result then follows because each of the sets $C_\alpha$, $C_\beta$, and $C_\gamma$ is closed under conjugation.
\end{proof}

For positive integers $n$ and $B > \ell N_A$, let $P(B,n)$ be the probability that $n$ primes $p_1,\ldots,p_n$ (possibly non-distinct) chosen uniformly at random in the interval $[B,2B]$ have the property that 
\begin{multline*}
     \rho_{A,\ell}(\Frob_{p_i}) \not\in C_\alpha \text{ for each $i$} \quad \text{or} \quad \rho_{A,\ell}(\Frob_{p_i}) \not\in C_\beta \text{ for each $i$} \quad \\ \text{or} \quad \rho_{A,\ell}(\Frob_{p_i}) \not\in C_\gamma \text{ for each $i$}.
\end{multline*}

\begin{cor}
\label{C:ntrials}
Suppose $C$ and $\ell$ are as in Lemma \ref{L:CebOnePrime} and let $n$ be a positive integer. For any $\epsilon > 0$, there exists an effective constant $B_0$ (with $B_0 > \ell N_A$) such that for all $B > B_0$, we have $$P(B,n) < (1-\alpha_\ell)^n  + (1-\beta_\ell)^n  + (1-\gamma_\ell)^n + \epsilon.$$
\end{cor}
\begin{proof}
For $\delta \in \{\alpha,\beta,\gamma\}$, let $X_\delta$ be the event that none of the $\rho_{A,\ell}(\Frob_{p_i})$ are contained in $C_\delta$. We then have 
\[
P(X_\alpha \cup X_\beta \cup X_\gamma) 
\le
 P(X_\alpha) + P(X_\beta) + P(X_\gamma) 
 \]
The result then follows by Lemma \ref{L:CebOnePrime}, which shows that there exists a $B_0$ such that the probabilities of $X_\alpha$, $X_\beta$, and $X_\gamma$ can be made arbitrarily close to $(1-\alpha_\ell)^n$, $(1-\beta_\ell)^n$, and $(1-\gamma_\ell)^n$ respectively.
\end{proof}

\begin{proof}[Proof of Theorem \ref{T:Bprob}]
The claim made by Theorem \ref{T:Bprob} is that $P(B,n) < 3 \cdot \left(\frac{9}{10}\right)^n$ for $B$ sufficiently large. By Proposition \ref{P:odds}, we have $1 - \alpha_\ell \le \frac{4}{5}$, $1 - \beta_\ell \le \frac{7}{8}$, and $1 - \gamma_\ell \le \frac{9}{10}$ for all $\ell$ odd. The result then follows from Corollary \ref{C:ntrials} because $\left(\frac{4}{5}\right)^n + \left(\frac{7}{8}\right)^n + \left(\frac{9}{10}\right)^n < 3 \cdot \left(\frac{9}{10} \right)^n$.
\end{proof}

\section{Results of computation and interesting examples}\label{S:examples}

We report on the results of running our algorithm on a dataset of $1{,}743{,}737$ typical genus $2$ curves with conductor bounded by $2^{20}$ which are part of a new dataset of approximately $5$ million curves currently being prepared for addition into the LMFDB. Running our algorithm on all of these curves in parallel took about $35$ hours on MIT's Lovelace computer (see the Introduction for the hardware specification of this machine). Instructions for obtaining the entire results file may be found in the \path{README.md} file of the repository.

We first show in Table~\ref{tab:nonsurj_primes} how many of these curves were nonsurjective at particular primes, indicating also if this can be explained by the existence of a rational torsion point of that prime order. We found $31$ as the largest nonsurjective prime, which occurred for the curve 
\begin{equation}\label{equation:nonsurj_31}
y^2 + (x + 1)y = x^5 + 23x^4 - 48x^3 + 85x^2 - 69x + 45
\end{equation}
of conductor $7^2 \cdot 31^2$ and discriminant $7^2 \cdot 31^9$ (the prime $2$ was also nonsurjective here). The Jacobian of this curve does not admit a nontrivial rational $31$-torsion point, so unlike many other instances of nonsurjective primes we observed, this one cannot be explained by the presence of rational torsion. One could ask if it might be explained by the existence of a $\qq$-rational $31$-isogeny (as suggested by Algorithm~\ref{A:combinedPart1algo}, since \(31\) is returned by Algorithm~\ref{alg:related}). This seems to be the case - see forthcoming work of van Bommel, Chidambaram, Costa, and Kieffer \cite{chidambaram2022computing} where the isogeny class of this curve (among others) is computed.

\begin{table}[h]
\renewcommand{\arraystretch}{1.25}
\begin{tabular}{|c|c|c|c|}
\hline
nonsurj.~prime & \# w/~torsion  &\# w/o~torsion & Example curve\\
\hline
\(2\) & $1{,}060{,}966$ & $437{,}201$ & \href{https://lmfdb.org/Genus2Curve/Q/464/a/464/1}{464.a.464.1}\\
\(3\) & $76{,}265$ & $95{,}108$ & \href{https://lmfdb.org/Genus2Curve/Q/277/a/277/2}{277.a.277.2}\\
\(5\) & $11{,}365$ & $10{,}044$ & \href{https://lmfdb.org/Genus2Curve/Q/16108/b/64432/1}{16108.b.64432.1}\\
\(7\) & $1{,}857$ & $2{,}056$ & \href{https://lmfdb.org/Genus2Curve/Q/295/a/295/2}{295.a.295.2}\\
\(11\) & $162$ & $203$ & \href{https://lmfdb.org/Genus2Curve/Q/4288/b/548864/1}{4288.b.548864.1}\\
\(13\) & $106$ & $261$ & \href{https://lmfdb.org/Genus2Curve/Q/439587/d/439587/1}{439587.d.439587.1}\\
\(17\) & $22$ & $51$ & \href{https://lmfdb.org/Genus2Curve/Q/1996/b/510976/1}{1996.b.510976.1}\\
\(19\) & $10$ & $20$ & 1468.6012928\\
\(23\) & $2$ & $8$ & 6784.1821066133504\\
\(29\) & $1$ & $5$ & 79056.59014987776\\
\(31\) & $0$ & $1$ & 47089.1295541485872879\\
\hline
\end{tabular}
\vspace{3pt}
\caption{\label{tab:nonsurj_primes}Nonsurjective primes in the dataset, and whether they are explained by torsion, with examples from the LMFDB dataset if available, else a string of the form ``conductor.discrimnant''.}
\end{table}

We also observed (see Table~\ref{tab:hist_nonmax_count}) that the vast majority of curves had less than $3$ nonsurjective primes.

\begin{table}[h]
\renewcommand{\arraystretch}{1.25}
\begin{tabular}{|c|c|c|c|}
\hline
\# nonsurj.~primes & \# curves & Example curve & Nonsurj.~primes of example\\
\hline
$0$ & $199{,}183$ & \href{https://lmfdb.org/Genus2Curve/Q/743/a/743/1}{743.a.743.1} & --\\
$1$ & $1{,}394{,}671$ & \href{https://lmfdb.org/Genus2Curve/Q/1923/a/1923/1}{1923.a.1923.1} & 5(torsion)\\
$2$ & $148{,}606$ & \href{https://lmfdb.org/Genus2Curve/Q/976/a/999424/1}{976.a.999424.1} & 2, 29(torsion)\\
$3$ & $1{,}277$ & \href{https://lmfdb.org/Genus2Curve/Q/15876/a/15876/1}{15876.a.15876.1} & 2, 3, 5\\\hline
\end{tabular}
\vspace{3pt}
\caption{\label{tab:hist_nonmax_count}Frequency count of nonsurjective primes in the dataset, with examples from the LMFDB dataset.}
\end{table}

It is interesting to compare Tables~\ref{tab:nonsurj_primes} and \ref{tab:hist_nonmax_count} to the analogous tables for non-CM elliptic curves over $\qq$ ($3{,}816{,}674$ curves), which are Tables~\ref{tab:nonsurj_primes_ecq} and \ref{tab:hist_nonmax_count_ecq} respectively (we omit example curves here since they can be readily searched for in the LMFDB).

\begin{table}[h]
\renewcommand{\arraystretch}{1.25}
\begin{tabular}{|c|c|c|}
\hline
nonsurj.~prime & \# w/~torsion  & \# w/o~torsion\\
\hline
\(2\) & $1{,}332{,}490$ & $5{,}726$\\
\(3\) & $57{,}930$ & $213{,}654$\\
\(5\) & $1{,}545$ & $19{,}211$\\
\(7\) & $80$ & $4{,}100$\\
\(11\) & $0$ & $156$\\
\(13\) & $0$ & $736$\\
\(17\) & $0$ & $40$\\
\(37\) & $0$ & $96$\\
\hline
\end{tabular}
\vspace{3pt}
\caption{\label{tab:nonsurj_primes_ecq}Nonsurjective primes for non-CM elliptic curves over $\qq$ in the LMFDB, and whether they are explained by torsion.}
\end{table}

\begin{table}[h]
\renewcommand{\arraystretch}{1.25}
\begin{tabular}{|c|c|}
\hline
\# nonsurj.~primes & \# curves\\
\hline
$0$ & $2{,}233{,}530$\\
$1$ & $1{,}530{,}524$\\
$2$ & $52{,}620$\\
\hline
\end{tabular}
\vspace{3pt}
\caption{\label{tab:hist_nonmax_count_ecq}Frequency count of nonsurjective primes for non-CM elliptic curves over $\qq$ in the LMFDB.}
\end{table}

We observe a similar pattern regarding the majority of curves nonsurjective at $2$ being explained by torsion, though in the elliptic curve case a much larger proportion are explained by $2$-torsion than for genus $2$ curves. This switches for the nonsurjective prime $3$ in both cases, although again for elliptic curves, the discrepancy is much starker. The zeroes in the torsion column for Table~\ref{tab:nonsurj_primes_ecq} are explained by Mazur's torsion theorem. The number of nonsurjective primes between genus 1 and genus 2 is qualitatively different: the majority of elliptic curves do not have any nonsurjective primes, while the vast majority of genus 2 curves have precisely one nonsurjective prime. It is also curious that the elliptic curve dataset does not contain a curve with $3$ nonsurjective primes.

We conclude with a few examples that illustrate where Algorithm~\ref{A:combinedPart1algo} fails when the abelian surface has extra (geometric) endomorphisms.

\begin{example}
The Jacobian \(A\) of the genus \(2\) curve \href{https://www.lmfdb.org/Genus2Curve/Q/3125/a/3125/1}{3125.a.3125.1} on the LMFDB given by $y^2 + y = x^5$ has \(\End(A_\qq) = \zz\) but \(\End(A_{\bar{\qq}}) = \zz[\zeta_5]\).
Let \(\phi\) be the Dirichlet character of modulus \(5\) defined by the Legendre symbol
\[\phi \colon (\zz/5\zz)^\times \to \{\pm 1\}, \qquad 2 \mapsto -1.\]
In this case, Algorithm~\ref{alg:quad_char} fails to find an auxilliary prime \(p < 1000\) for which \(a_p \neq 0\) and \(\phi(p) = -1\).
This is consistent with the endomorphism calculation, since the trace of \(\rho_{A, \ell}(\Frob_p)\) is \(0\) for all primes \(p\) that do not split completely in \(\qq(\zeta_5)\) and any inert prime in \(\qq(\sqrt{5})\) automatically does not split completely in \(\qq(\zeta_5)\).
\end{example}

\begin{example}
The modular curve \(X_1(13)\) (\href{https://www.lmfdb.org/Genus2Curve/Q/169/a/169/1}{169.a.169.1}) has genus \(2\) and its Jacobian \(J_1(13)\) has CM by \(\zz[\zeta_3]\) over \(\qq\).
As in \cite[Claim 2, page 45]{MazurTate}, for any prime \(\ell\) that splits as \(\pi \bar{\pi}\) in \(\qq(\zeta_3)\), the representation \(J_1(13)[\ell]\) splits as a direct sum \(V_{\pi} \oplus V_{\bar{\pi}}\) of two \(2\)-dimensional subrepresentations that are dual to each other.  (A similar statement holds for \(J_1(13)[\ell] \otimes_{\ff_\ell} \bar{\ff}_\ell\), and so this representation is never absolutely irreducible.)  As expected, Algorithm~\ref{alg:related} fails to find an auxiliary prime \(p < 1000\) for which \(R_p\) is nonzero.
\end{example}

\begin{example}
The first (ordered by conductor) curve whose Jacobian $J$ admits real multiplication over $\overline{\qq}$ is the curve \href{https://www.lmfdb.org/Genus2Curve/Q/529/a/529/1}{529.a.529.1}; indeed, this Jacobian is isogenous to the Jacobian of the modular curve $X_0(23)$. Since there is a single Galois orbit of newforms - call it $f$ - of level $\Gamma_0(23)$ and weight $2$, we have that $J$ is isogenous to the abelian variety $A_f$ associated to $f$, and thus we expect the integer $M_{\text{self-dual}}$ output by Algorithm~\ref{alg:cuspforms} to be zero for any auxiliary prime, which is indeed the case.
\end{example}

\bibliographystyle{amsalpha}
\bibliography{galrepabsurf}

\newpage

\appendix

\section{Exceptional maximal subgroups of 
\texorpdfstring{\(\GSp_4(\ff_\ell)\)}{GSp4Fell}}

{\small
\begin{table}[h!]\label{T:excgrpgen}
    \begin{tabular}{c|c|c|c}
        \(\ell\) & type & choices & generators \\ \hline
         \(\ell \equiv 5 \bmod{8}\) & \(G_{1920}\) & \(b^2=-1\) in \(\ff_\ell\) & \begin{tabular}{@{}c@{}}\( \begin{pmatrix} 1& 0& 0& -1\\ 0& 1& -1& 0 \\ 0& 1& 1& 0 \\ 1& 0& 0&1 \end{pmatrix}\), \(\begin{pmatrix} 1& 0& 0& b\\0& 1& b& 0 \\ 0 & b & 1 & 0 \\ b & 0 & 0 & 1 \end{pmatrix}\), \\[2pt] \(\begin{pmatrix} 1 & 0 & 0 & -1 \\ 0 & 1 & 1 & 0 \\ 0 & -1& 1& 0 \\ 1& 0& 0& 1 \end{pmatrix}\), \(\begin{pmatrix} 1& 0& 1& 0 \\ 0& 1& 0& 1 \\ -1& 0& 1& 0 \\ 0& -1& 0& 1\end{pmatrix}\) \end{tabular} \\\hline
         \(\ell \equiv 3 \bmod{8}\) & \(G_{1920}\) & \(b^2=-2\) in \(\ff_\ell\) & \begin{tabular}{@{}c@{}}\( \begin{pmatrix} 1& 0& 0& -1\\ 0& 1& -1& 0 \\ 0& 1& 1& 0 \\ 1& 0& 0&1 \end{pmatrix}\), \(\begin{pmatrix} 0& 0& 0& b\\0& 0& b& 0 \\ 0 & b & 2 & 0 \\ b & 0 & 0 & 2 \end{pmatrix}\), \\[2pt] \(\begin{pmatrix} 1 & 0 & 0 & -1 \\ 0 & 1 & 1 & 0 \\ 0 & -1& 1& 0 \\ 1& 0& 0& 1 \end{pmatrix}\), \(\begin{pmatrix} 1& 0& 1& 0 \\ 0& 1& 0& 1 \\ -1& 0& 1& 0 \\ 0& -1& 0& 1\end{pmatrix}\)  \end{tabular} \\\hline
         \(\ell \equiv 7 \bmod{12}\) & \(G_{720}\) & \(a^2+a+1=0\) in \(\ff_\ell\) & \begin{tabular}{@{}c@{}} \(\begin{pmatrix}a& 0& 0& 0 \\ 0& a& 0& 0 \\ 0& 0& 1& 0 \\ 0 & 0 & 0& 1\end{pmatrix}\), \(\begin{pmatrix} a& 0& 0& 0 \\ 0 & 1 & 0 & 0 \\0 & 0& a& 0 \\ 0 & 0 & 0 & 1 \end{pmatrix}\), \\[2pt] \(\begin{pmatrix} a& 0& -a-1& a+1 \\ 0& a& -a-1& -a-1 \\ -a-1& -a-1& -1& 0 \\ a+1& -a-1& 0& -1 \end{pmatrix}\), \(\begin{pmatrix}0& -1& 0& 0 \\ 1& 0& 0& 0 \\ 0&  0& 0& -1 \\ 0& 0& 1& 0\end{pmatrix}\) \end{tabular} \\\hline
         \(\ell \equiv 5 \bmod{12}\) & \(G_{720}\) & \(b^2=-1\) in \(\ff_\ell\) & \begin{tabular}{@{}c@{}} \(\begin{pmatrix} -1& 0& 0& -1 \\ 0& -1& -1& 0 \\ 0 & 1& 0& 0 \\ 1 & 0& 0& 0 \end{pmatrix}\), \(\begin{pmatrix} 0& 0& 0& 1 \\ 0 & -1 & -1& 0 \\ 0 & 1& 0& 0 \\ -1& 0& 0& -1 \end{pmatrix}\),\\[2pt] \(\begin{pmatrix} -b-1 & b & 2b & -2b+1 \\ b& b-1& 2b+1& 2b \\ b & b-1& -b-2& -b \\ -b-1& b& -b& b-2 \end{pmatrix}\), \(\begin{pmatrix} 0& -b& -2b& 0\\ b& 0& 0& 2b \\ -2b & 0 & 0 & -b \\ 0 & 2b &  b & 0 \end{pmatrix}\)  \end{tabular} \\ \hline
          \(\ell = 7\) & \(G_{5040}\) & \begin{tabular}{@{}c@{}}\(a=2\) satisfies \\[2pt]  \(a^2+a+1=0\)\end{tabular}  & \begin{tabular}{@{}c@{}} \(\begin{pmatrix}2& 0& 0& 0 \\ 0& 2& 0& 0 \\ 0& 0& 1& 0 \\ 0 & 0 & 0& 1\end{pmatrix}\), \(\begin{pmatrix} 2& 0& 0& 0 \\ 0 & 1 & 0 & 0 \\0 & 0& 2& 0 \\ 0 & 0 & 0 & 1 \end{pmatrix}\), \\[2pt] \(\begin{pmatrix} 6 & 0 & 5 & 2 \\ 0 & 6 & 5 & 5 \\ 5 & 5 & 4 & 0 \\ 2 & 5 & 0 & 4 \end{pmatrix}\), \(\begin{pmatrix}0& 6& 0& 0 \\ 1& 0& 0& 0 \\ 0&  0& 0& 6 \\ 0& 0& 1& 0\end{pmatrix}\), \(\begin{pmatrix} 4 & 6 & 0 & 0 \\ 6& 6& 0& 0 \\ 0 & 0 & 4 & 1 \\ 0 & 0 & 1 & 6 \end{pmatrix}\)  \end{tabular}
    \end{tabular}
    \vspace{10pt}
    \caption{Explicit generators for each exceptional maximal subgroup in \(\GSp_4(\ff_\ell)\) (up to conjugacy).  The matrices described in Table~\ref{tab:exceptional} depend on an auxiliary choice of a parameter denoted either $a$ and $b$ in each case. In each row, any one choice of the corresponding $a$ and $b$ satisfying the equations described in the table suffices.}\label{tab:exceptional}
\end{table}

}


\end{document}